\documentclass[a4paper, 10pt, twoside, notitlepage]{amsart}
\usepackage[utf8]{inputenc}
\usepackage{color}
\usepackage{amsmath} 
\usepackage{amssymb} 
\usepackage{amsthm}
\usepackage{graphicx}
\usepackage{esint}
\usepackage[colorlinks=true,linkcolor=blue]{hyperref}
\usepackage[export]{adjustbox}
\usepackage{thmtools}
\usepackage{tikz}
\usepackage{braket}
\usepackage{hyphenat}

\usepackage{enumitem}
\usepackage{subcaption}

\theoremstyle{plain}
\newtheorem{thm}{Theorem}
\newtheorem{prop}{Proposition}[section]

\newtheorem{lem}[prop]{Lemma}
\newtheorem{cor}[prop]{Corollary}

\newtheorem{defi}[prop]{Definition}
\newtheorem{rmk}[prop]{Remark}

\newtheorem{alg}[prop]{Algorithm}

\newcommand {\R} {\mathbb{R}} \newcommand {\Z} {\mathbb{Z}}
 \newcommand {\N} {\mathbb{N}}
 
\newcommand {\p} {\partial}

\newcommand {\Cof} {\text{Cof}}

\newcommand {\supp} {\text{supp}}

\newcommand{\lbr}{\langle}
\newcommand{\rbr}{\rangle}

\newcommand{\vc}[1]{\mathbf{#1}}
\newcommand {\eps} {\varepsilon}
\newcommand{\mt}[1]{\mathsf{#1}}

\DeclareMathOperator{\inter}{int}
\DeclareMathOperator{\rank}{rank}

\newcommand{\beq}[0]{\begin{equation}}
\newcommand{\eeq}[0]{\end{equation}}

\DeclareMathOperator{\tr}{tr}

\DeclareMathOperator {\dist} {dist}

\DeclareMathOperator {\conv}{conv}
\DeclareMathOperator {\inte} {int}

\DeclareMathOperator {\Per} {Per}
\DeclareMathOperator {\sgn} {sgn}

\begin{document}

\title[The two-well problem]{Convex Integration Solutions for the Geometrically Non-linear Two-Well Problem with Higher Sobolev Regularity}

\author[Francesco Della Porta]{Francesco Della Porta}
\address{Max-Planck-Institute for Mathematics in the Sciences, Inselstr. 22, 04103 Leipzig}
\email{Francesco.DellaPorta@mis.mpg.de}

\author[Angkana Rüland]{Angkana Rüland}
\address{Max-Planck-Institute for Mathematics in the Sciences, Inselstr. 22, 04103 Leipzig}
\email{rueland@mis.mpg.de}

\begin{abstract}
In this article we discuss higher Sobolev regularity of convex integration solutions for the geometrically non-linear two-well problem. More precisely, we construct solutions to the differential inclusion $\nabla\vc u\in K$ subject to suitable affine boundary conditions for $\vc u$ with
$$
K:= SO(2)\left[\begin{array}{ ccc } 1 & \delta \\ 0 & 1  \end{array}\right] \cup SO(2)\left[\begin{array}{ ccc } 1 & -\delta \\ 0 & 1  \end{array}\right]
$$
such that the associated deformation gradients $\nabla\vc u$ enjoy higher Sobolev regularity. This provides the first result in the modelling of phase transformations in shape-memory alloys where  $K^{qc} \neq K^{c}$, and where the energy minimisers constructed by convex integration satisfy higher Sobolev regularity.
We show that in spite of additional difficulties arising from the treatment of the non-linear matrix space geometry, it is possible to deal with the geometrically non-linear two-well problem within the framework outlined in \cite{RZZ18}. Physically, our investigation of convex integration solutions at higher Sobolev regularity is motivated by viewing regularity as a possible selection mechanism of microstructures.
\end{abstract}

\maketitle

\section{Introduction}

Shape-memory alloys are solid materials which undergo a first order diffusionless solid-solid phase transformation in which symmetry is reduced in passing from the high temperature phase (\emph{austenite}) to the low temperature phase (\emph{martensite}). This reduction of symmetry gives rise to complex microstructures. Mathematically, these have been successfully modelled within the framework of the calculus of variations \cite{BJ92, Ball:ESOMAT} by minimizing energy functionals of the form
\begin{align}
\label{eq:energy_min}
\int\limits_{\Omega} W(\nabla \vc u, \theta) d\vc x.
\end{align}
Here $\Omega$ denotes a reference domain, $W: \R^{n\times n}_{+} \times (0,\infty) \rightarrow  [0,\infty)$ models a stored energy function which is typically highly non-quasiconvex, $\vc u: \Omega \rightarrow \R^n$ describes the deformation of the material and $\theta:\Omega \rightarrow (0,\infty)$ denotes temperature. In order to model the solid-solid phase transformation under consideration, the energy density is assumed to be 
\begin{itemize}
\item \emph{frame indifferent}, i.e., $W(\mt Q \mt F) = W(\mt F)$ for all $\mt Q\in SO(n)$,
\item \emph{invariant under material symmetries}, i.e., $W(\mt F \mt H) = W(\mt F)$ for all $\mt H \in \mathcal{P}$, where the group $\mathcal{P} \subset SO(n)$ models the material symmetries.
\end{itemize}
As the minimization of \eqref{eq:energy_min} (subject, e.g., to certain boundary conditions) is typically very difficult, a common first step towards the analysis of possible microstructures is the analysis of exactly \emph{stress-free solutions}, i.e., of deformations $\vc u: \Omega \rightarrow \R^2$ such that $W(\nabla \vc u, \theta) = 0$, which can also be rephrased as 
\begin{align}
\label{eq:diff_inclusion}
\nabla \vc u \in K(\theta).
\end{align}
The set $K(\theta)$ models the \emph{energy wells}, i.e., the zero energy states. By frame-indifference and material symmetry these are of the form
\begin{align*}
K(\theta):= 
\left\{
\begin{array}{ll}
\alpha(\theta) SO(n) \mbox{ for } \theta > \theta_c,\\
SO(n) \mbox{ for } \theta = \theta_c,\\
\bigcup\limits_{j=1}^{m} SO(n) \mt{U}_j \mbox{ for } \theta< \theta_c,
\end{array} \right.
\end{align*}
where $\mt{U}_j^T = \mt{U}_j \in \R^{n\times n}$ and $\mt{U}_j = \mt{P} \mt{U}_1 \mt{P}^T$ for some $\mt{P} \in \mathcal{P}$ are positive definite matrices modelling the \emph{variants of martensite}.
Here $\theta_c>0$ denotes the transformation temperature from the high temperature phase, \emph{austenite}, to the low temperature phase, \emph{martensite}.

In solving \eqref{eq:diff_inclusion} a surprising dichotomy arises. More precisely, there are physically relevant models for which on the one hand the problem \eqref{eq:diff_inclusion} is very \emph{rigid} if regularity conditions (which physically model surface energy constraints) are imposed. If for instance $m=2$, 
these further regularity assumptions force the solutions to \eqref{eq:diff_inclusion} to obey non-linear ``hyperbolic'' partial differential equations, and the solutions to propagate along characteristics, thus leading to rigidity of the differential inclusion
 (see \cite{DM1} and also \cite{K} for the case $m=3$ and \cite{R16} for the case $m=6$ in the geometrically linearised setting in three dimensions). If on the other hand, no regularity assumptions are made, then in many models a plethora of ``wild'' solutions exist \cite{MS, DaM12}; the differential inclusion \eqref{eq:diff_inclusion} becomes very flexible, the notion of characteristics is lost. 
While for a number of physical problems these ``end point cases'' are understood, a theory describing the ``transition'' from the flexible regime of ``wild'' solutions to the regime of ``rigid'' solutions is not yet available for martensitic transformations. Only very recently first results on the persistence of wild solutions at low, but positive regularity have been obtained \cite{RZZ16, RZZ18, RTZ19}. It is the purpose of this note to further study the problem in a truly ``geometrically non-linear'' regime.

In \cite{RZZ18} the authors provided a general framework for deducing higher regularity of convex integration solutions and applied this to a number of relevant phase transformation problems from the materials sciences. However, all the examples in \cite{RZZ18} are limited to the setting where $K^{qc} = K^{c}$. If this is not the case, additional difficulties arise from understanding potentially complicated matrix space geometries (e.g., the 
applications of the algorithm in \cite{RZZ18} heavily use barycentric coordinates in defining associated in-approximations). 
Discussing higher regularity convex integration solutions for the geometrically non-linear two-well problem in this article, we provide the first \emph{geometrically non-linear} application of the algorithm in \cite{RZZ18} in which $K^{qc}$ is \emph{strictly} smaller than $K^{c}$.  

To this end, we fix a temperature $\theta$ below the critical temperature and study the two-dimensional two-well problem for which 
\beq
\label{defK_intro}
K = SO(2)\mt F_0\cup SO(2)\mt F_0^{-1}
\eeq
where $\mt F_0,\mt F_0^{-1}\in\R^{2\times2}$ are respectively given by
\begin{align}
\label{eq:matrices_intro}
{\mt F}_0 = \left[\begin{array}{ ccc } 1 & \delta \\ 0 & 1  \end{array}\right],\qquad
{\mt F}_0^{-1} = \left[\begin{array}{ ccc } 1 & -\delta \\ 0 & 1  \end{array}\right],
\end{align}
and $\delta>0.$ We remark that, as shown in \cite[Sec. 5]{BJ92}, given two wells with two rank-one connections (and the physically natural condition of equal determinant), one can always reduce the problem to our case via an affine change of variables. Important properties of this phase transformation are:
\begin{itemize}
\item \emph{A large lamination convex hull:} $K^{lc} = K^{pc} = K^{c} \cap \{\mt F: \ \det(\mt F) = 1\}$ (see \eqref{eq:hulls}) and thus $\dim(K^{lc}) = 3$,
\item \emph{A dichotomy between rigidity and flexibility.} On the one hand, the two-well problem $\nabla \vc u \in K$ with $K$ as in \eqref{defK} allows for convex integration solutions \cite{MS, D, DaM12}, that menas, it is flexible if no regularity is imposed, i.e., if $\nabla \vc u $ is only bounded. On the other hand, the phase transformation is rigid if additional surface energy constraints are imposed: If $\nabla \vc u \in BV$, then a solution to the differential inclusion is locally a simple laminate (see \cite{DM1}). 
\end{itemize}
Thus, the two-well problem is one of the simplest geometrically non-linear models for martensitic phase transformations in which a dichotomy between rigidity and flexibility exists. It hence provides an ideal model setting for studying the transition between these regimes.

We remark that a physical motivation to understand the dichotomy between rigidity and flexibility is the characterisation of physically relevant microstructures. In this context, regularity could provide a selection mechanism for minimisers of \eqref{eq:energy_min} (see \cite{RTZ19}). A different approach to select physically relevant microstructures other than regularity can be found in \cite{FDP1}. We refer the reader to the discussion after Theorem \ref{thm:main} for more details.

\subsection{Main results}
In the setting of the two-well problem a main difficulty and novelty that we address is the fact that $K^{lc} \neq K^{c}$. Hence, a key aspect of our higher regularity analysis of convex integration solutions to the differential inclusion \eqref{eq:diff_inclusion} for \eqref{defK_intro} and \eqref{eq:matrices_intro} involves an explicit investigation of the matrix space geometry of the two-well problem. Seeking to embed the higher regularity question of convex integration solutions for the two-well problem into the framework of \cite{RZZ18}, we rely on an explicit double-laminate construction (which is inspired by the construction in \cite[Chapter 5.3]{P}) which yields suitable ``orthogonal'' coordinates for $K^{qc}$.

Exploiting the precise understanding of the quasiconvex hull $K^{qc}$, we obtain the following main result:

\begin{thm}
\label{thm:main}
Let $K$ be as in \eqref{defK_intro}--\eqref{eq:matrices_intro}. Let $\Omega \subset \R^2$ satisfy 
\begin{equation}
\label{eqDomain}
\tag{D}
\text{\parbox{3.8in}{\centering $\Omega$ is open, connected, and can be covered (up to a set of measure zero) by finitely many open disjoint triangles.}}
\end{equation}
Then there exists $\theta_0 >0$ such that for all $s\in (0,1)$, $p\in (1,\infty)$ with $sp < \theta_0$ and for all $\mt M \in \inter K^{qc}$ there exists a deformation $\vc u:\Omega \rightarrow \R^2$ such that
\begin{align*}
\nabla \vc u &\in K \mbox{ a.e. in } \Omega, \\
\vc u & = \mt M \vc x \mbox{ on } \partial \Omega,\\ 
\vc u &\in W^{1,\infty}(\Omega;\R^{2})\cap W^{1+s,p}(\Omega;\R^{2}).
\end{align*}
Further, for $j\in \{1,2\}$ we also have $\chi^j \in W^{s,p}(\Omega) $, where 
\begin{align*}
\chi^1(x):= 
\left\{
\begin{array}{ll}
1 \mbox{ if } (\nabla \vc u^T \nabla \vc u)(x) = \mt F_0^T \mt F_0,\\
0 \mbox{ else},  
\end{array}
\right.
\end{align*}
and $\chi^2 = 1- \chi^1$.
\end{thm}

\begin{rmk}
\label{rmk:D}
Arguing as in \cite{RZZ16, RZZ18}, it would have been possible to relax the condition \eqref{eqDomain}. As in the present article this is \emph{not} our main emphasis, we do not present the details for this possible extension of our results, but refer to the arguments in \cite{RZZ16, RZZ18}.
\end{rmk}

\begin{rmk}
\label{rmk:D1}
Arguing as in \cite{MS}, it would have been possible to generalise the result (at least) to some finitely piecewise affine boundary data supported on piecewise polygonal domains 
or extend the problem to higher dimensions. We refer to \cite{FDP2} for a rigidity result for the two-well problem in three dimensions with piecewise affine boundary conditions.
\end{rmk}

Let us comment on some of the main new ingredients leading to the construction of convex integration solutions with higher Sobolev regularity in the geometrically non-linear two-well problem:
\begin{itemize}
\item \emph{Matrix space geometry and $K^{qc}$.}
In order to construct convex integration solutions to \eqref{eq:diff_inclusion}--\eqref{eq:matrices_intro} which enjoy higher Sobolev regularity a central step is the construction of suitable ``orthogonal'' coordinates for $K^{qc}$. These allow us to move along a rank-one line in $K^{qc}$ when varying one of the coordinates. In order to avoid dealing with rotations, we investigate this in the Cauchy-Green space $\mathcal{G}$ associated with $K^{qc}$, that is the set of symmetric positive definite matrices $\mt F^T\mt F$ for $\mt F\in K^{qc}$. This further has the advantage that the projection of $K^{qc}$ onto Cauchy-Green space is (a subset of) a two-dimensional manifold which is well-understood. However, some attention must be paid to relate objects in $K^{qc}$ to objects in its Cauchy-Green space. In particular, in order to exploit our self-improving replacement construction (see Lemma \ref{lem:Conti}), we rely on Lipschitz bounds in order to pass from relevant quantities in matrix space to their projections in Cauchy-Green space (see Lemmas \ref{lemma2}, \ref{lemma3}). For our convex integration algorithm (see Section \ref{sec:Algorithm}), it is important that all dependences are explicit in order to achieve the higher Sobolev regularity of $\nabla\vc u$. Thus, a good understanding of the geometry of $K^{qc}$ is crucial (see Section \ref{sec:matrix_space_geo} below).
\item \emph{A quantitative replacement construction.}
 Another important ingredient is the replacement construction (see Lemma \ref{lem:Conti} below) which we use at every iteration step of the convex integration algorithm to improve on the present gradient distribution and thus to move $\nabla\vc u$ closer to $K$ in some sub-region of our domain. Here, we rely on an explicit construction from \cite{C} which we complement with explicit quantitative bounds, necessary to prove the estimates of higher Sobolev regularity of $\nabla\vc u$ following the algorithm in \cite{RZZ18}. 
\end{itemize}
With the these two ingredients in hand, it is possible to embed the geometrically non-linear two-well problem into the framework of \cite{RZZ18} and to thus deduce the existence of convex integration solutions with higher Sobolev regularity for the geometrically non-linear two-well problem for the first time. As in \cite{RZZ18} we here do not optimize the constants in our construction but rather focus on the main qualitative dependences, leading to good uniform dependences on the Sobolev exponents (see the discussion at the end of Section \ref{sec:lit}). As a consequence,  our regularity threshold $\theta_0>0$ can be chosen uniformly for affine boundary data in $K^{lc}$ (with norms that deteriorate for boundary data close to the boundary of $K^{lc}$).

\begin{rmk}
We remark that in contrast to the construction from \cite{RZZ16} we here rely on a convex integration scheme using a \emph{countably infinite} sequence of replacements (see the discussion in \cite{RTZ19} comparing a countably infinite with a finite scheme for the geometrically linearised hexagonal-to-rhombic transformation). An interesting alternative matrix space construction is given in  \cite{BJ92} which could constitute the core of a convex integration construction for which at almost every point in matrix space only \emph{finitely} many replacement steps are considered (with the number of replacements however depending on the point under consideration). In order to exploit this matrix space construction in a convex integration algorithm, one would however need to make use of a special replacement construction (cf. Lemma \ref{lem:Conti} below) at which a fixed endpoint is deformed ``non-perturbatively'' (see \cite{RZZ16} for such a construction in the case of the geometrically linearised hexagonal-to-rhombic phase transformation). As we are not aware of such a replacement construction which also allows one to preserve the determinant, we do not pursue this further here. For special settings in which highly symmetric replacement constructions exist, we refer to \cite{CDPRZZ19, CKZ17}.
\end{rmk}

Relying on an idea by Ball and James \cite{BJ92}, we complement our result of Theorem \ref{thm:main} by also showing that for the two-well problem with only \emph{one} rank-one connection no convex integration solutions exist for affine boundary conditions (see Proposition  \ref{lem:laminates}).

\subsection{Relation with the literature on the two-well problem}
\label{sec:lit}
Compared to other phase transformations which allow for the presence of convex integration solutions, the geometrically non-linear two-well problem is relatively well-studied (although still important questions -- in particular in the context of convex integration and scaling limits -- are open) which makes it an attractive model problem. In the context of our problem the most relevant known results are:
\begin{itemize}
\item \emph{Dichotomy rigidity-flexibility.} Due to the work \cite{MS, DM1, DaM12, DM96} it is known that in the limiting cases of $BV$ and $L^{\infty}$ regularity the differential inclusion \eqref{eq:diff_inclusion} with $K$ as in \eqref{defK_intro}, \eqref{eq:matrices_intro} is rigid for $BV$ regular solutions while it is highly flexible for merely $L^{\infty}$ regular solutions.
\item \emph{Convex hulls.} The convex hulls of the set $K$ are known and explicit (see \cite{Sverak,P}) which is crucial for our detailed matrix space analysis and the construction of ``orthogonal'' coordinates.
\item \emph{Scaling.} The scaling behaviour for minimisers of energies involving elastic and surface contributions is determined for $\mt M = \mt {Id}$ in \cite{CC14, ContiChan14a}. In particular, the arguments of \cite{CC14} persist if $\mt M \in \partial K^{lc}$ showing analogous scaling results for these boundary conditions.
While there are no convex integration solutions in this setting (see \cite{BJ92,P}), these scaling results do constitute a very important step into the further analysis of the scaling behaviour for the two-well problem (and suggest that the scaling and, associated to that, the convex integration regularity thresholds for boundary data $\mt M \in \inte(K^{lc})$ are probably not better than the ones for $\mt M \in \partial (K^{lc})$).

In \cite{Lorent06} the scaling of a continuum two-well model is linked to a corresponding discrete model showing that optimal bounds in one of these lead to optimal bounds in the other one.
\item \emph{Rigidity results.} Various rigidity type results (in the presence of small surface energies or in the setting of only one rank-one connection) are known \cite{JerrardLorent2013, ContiSchweizer06, ContiSchweizer06a, CC10, DavoliFriedrich18, KLLR19}.
\end{itemize}

In view of the results of \cite{RTZ19} and \cite{CC14}, in particular further scaling results would be of considerable interest as they could provide important hints at the optimal regularity of convex integration solutions and thus at regularity as a selection mechanism for minimizers of \eqref{eq:energy_min} in the presence of convex integration solutions.
However, in order to exploit this connection, it would be necessary to study the scaling behaviour functionals involving bulk and surface energies subject to affine boundary conditions such that $\mt M \in \inte(K^{lc})$ which we do not address here.
In this framework, our Theorem \ref{thm:main} provides an important first step towards understanding the dichotomy rigidity-flexibility for the two-well problem, and enlarges the function space where exact solutions to \eqref{eq:diff_inclusion} exist.

\subsection{Outline of the article}
The remainder of the article is organized as follows: In Section \ref{sec:matrix_space_geo} we discuss the matrix space geometry of the two-well problem and introduce our main coordinates (which we define in Cauchy-Green space). Based on this, in Sections \ref{sec:in_approx}-\ref{sec:Algorithm}, we combine the two-well matrix space geometry with the convex integration framework from \cite{RZZ18}. In Section \ref{sec:sing_set} we recall the connection between regularity estimates and (partial) upper and lower bounds for a suitably defined notion of dimension of the singular set of the deformation/ the underlying characteristic functions.
Relying on ideas from \cite{BJ92}, in Section \ref{sec:one_rank_one}, we show that in the two-well problem with only one rank-one connection no convex integration solutions exist.
Finally, for the convenience of the reader in the appendix, we collect the notation that we are using throughout the article.

\section{Matrix Space Geometry}
\label{sec:matrix_space_geo}

In this section we study the matrix space geometry of the two-well problem. As central observations we obtain a system of two rank-one connections which we will crucially exploit in our convex integration algorithm (see Lemmas \ref{lemma2}, \ref{lemma3}).

Consider the set 
\beq
\label{defK}
K = SO(2)\mt F_0\cup SO(2)\mt F_0^{-1}
\eeq
where $\mt F_0,\mt F_0^{-1}\in\R^{2\times2}$ are respectively given by
\begin{align}
\label{eq:matrices}
{\mt F}_0 = \left[\begin{array}{ ccc } 1 & \delta \\ 0 & 1  \end{array}\right],\qquad
{\mt F}_0^{-1} = \left[\begin{array}{ ccc } 1 & -\delta \\ 0 & 1  \end{array}\right],
\end{align}
and $\delta>0.$ We remark that, as shown in \cite[Sec. 5]{BJ92}, given two wells with two rank-one connections, one can always reduce the problem to our case via an affine change of variables. 
Moreover, we note that the two wells in \eqref{eq:matrices} can be transformed into the wells
\begin{align}
\label{eq:trafo}
\tilde{{\mt F}}_0 = \left[\begin{array}{ ccc } 1 & 0 \\ \bar{\delta} & 1  \end{array} \right] = \mt D \mt Q \mt F_0, \ 
\tilde{{\mt F}}_0^{-1} = \left[\begin{array}{ ccc } 1 & 0 \\ -\bar{\delta} & 1  \end{array}\right] = \mt D \mt Q \mt R \mt F_0^{-1},
\end{align}
where $\bar{\delta} = \frac{\delta}{1+\delta^2}$ and 
\begin{align*}
\mt Q = \left[\begin{array}{ ccc } 1 & -\delta \\\delta & 1  \end{array}\right], \ 
\mt D = \left[\begin{array}{ ccc } 1 & 0 \\ 0 & \frac{1}{1+\delta^2} \end{array}\right], \ 
\mt R = \frac{1}{1+\delta^2} \left[\begin{array}{ ccc } 1-\delta^2 & 2\delta \\ -2\delta & 1-\delta^2 \end{array}\right].
\end{align*}

\begin{rmk}
\label{SeparationProperty}
We notice that
$$
\dist^2(SO(2)\mt F_0,SO(2)\mt F_0^{-1}) = \min_{\mt R\in SO(2)}|\mt R \mt F_0 - \mt F_0^{-1}|^2 = 4+2\delta^2 -2\sqrt{\delta^4+4} >0
$$
for each $\delta>0.$ The wells are thus separated.
\end{rmk}

We are now interested in the set $K^{qc}$, where for any compact set $H$, $H^{qc}$ is given by
\beq
\label{def qc}
H^{qc} = \bigl\{\mt F\in\R^{2\times2}: f(\mt F)\leq \sup_{ H} f,\quad\text{for each $f\colon\R^{2\times 2}\to\R$ quasiconvex}	\bigr\},
\eeq
that is the set of constant macroscopic deformation gradients which can be obtained by finely mixing the martensitic variants $\mt F_0,\mt F_0^{-1}.$ We define $H^{lc}$ as $H^{lc} := \bigcup_{i=0}^\infty H^{(i)}$, where the sets $H^{(i)}$ are defined recursively by $H^{(0)} = H$ and
\begin{align*}
H^{(i)}&:= \left\{\mt M\in\R^{2\times 2}: \mt M = \lambda\mt A+(1-\lambda)\mt B,\quad\text{for any $\mt A,\mt B\in H^{(i-1)}$}, \right. \\
& \quad \quad \left. \text{$\rank(\mt A-\mt B)=1$, and $\lambda\in[0,1]$ }	\right\},
\end{align*}
for any $i\geq 1$, for any compact set $H\subset\R^{2\times 2}.$ When $K$ is defined as in \eqref{defK}, we have that \cite{BJ92,Sverak} 
\begin{align}
\label{eq:hulls}
K^{qc} = K^{lc} = K^c\cap \bigl\{\mt F\in\R^{2\times2}: \det \mt F = 1\bigr\},
\end{align}
where $K^c$ is the convex hull of $K$.

\subsection{Coordinates in Cauchy-Green space}
The following result from \cite{BJ92,P} gives more precise information on $K^{qc}$ (see also Figure \ref{fig:fig2}):

\begin{prop}[\cite{P}, Chapter 5.3]
\label{prop:Kqc}
The set of macroscopically attainable boundary conditions $\mt F\in \R^{2\times 2}$ for $K$ as in \eqref{defK} is characterized by the conditions
\begin{align*}
\mt C(\mt F) \in \mathcal{G}:=\{\mt C\in \R^{2\times 2}_{sym}: \ 
c_{11} \leq 1, \ c_{22} \leq (1+\delta^2) ,\ c_{11} c_{22} \geq 1, \ \det(\mt C)=1\}.
\end{align*}
Here $ \mt C(\mt F) =\mt F^T \mt F$ is the Cauchy-Green tensor associated with the macroscopic deformation $\mt F$ and $c_{ij}$ denote its components.
\end{prop}

We remark that by the determinant constraint $c_{12}^2= 1-c_{11}c_{22}$.

This result relies on two arguments: The necessary condition is determined by a discussion of the necessary condition for gradient Young measures which are supported in $K$. This in particular uses the fact that the determinant is weakly continuous.

The second step yielding sufficiency and thus attainability relies on an explicit construction of a suitable double laminate.

\begin{figure}[t]
  \centering
  \includegraphics[width=0.7\linewidth,page=7]{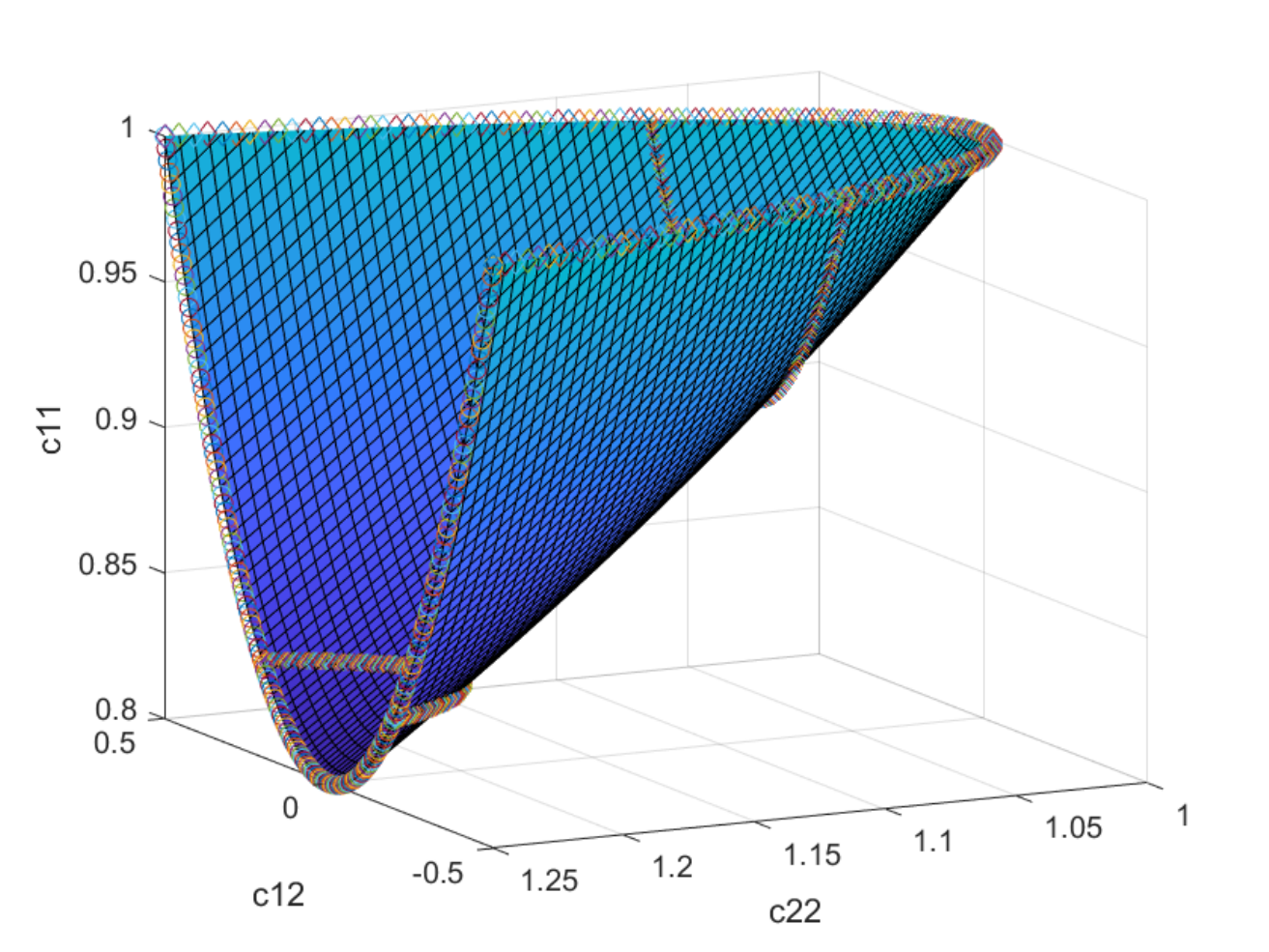}
\caption{Plotting $\mathcal G$ in case $\delta=0.5$}
\label{fig:fig2}
\end{figure}

Below we denote by $\inter K^{qc}$ the relative interior of $K^{qc}$ in the manifold $\Sigma:=\bigl\{\mt F\in\R^{2\times2}: \det \mt F = 1\bigr\}$. We notice that the two wells $SO(2)\mt F_0$, $SO(2)\mt F_0^{-1}$ are rank-one connected in the following way:
\beq
\label{Rk10}
\mt F_0^{-1} = \mt F_0 - 2\delta \vc e_1\otimes \vc e_2, \qquad 
\mt R \mt F_0^{-1} = \mt F_0 - \frac{2\delta}{1+\delta^2}(\delta \vc e_1 + \vc e_2)\otimes \vc e_1 ,
\eeq
where $\mt R= \frac{1}{1+\delta^2}\left[\begin{array}{ ccc } 1-\delta^2 & 2\delta \\ -2\delta & 1-\delta^2  \end{array}\right]\in SO(2)$.
We remark that the second rank-one connection corresponds to the one coming from $\tilde{\mt F}_0^{-1} = \tilde{\mt F}_0 - 2\bar{\delta} \vc e_2 \otimes \vc e_1  $ (in this sense the two rank-one connections are ``symmetric'').
Let us hence define
$$
\mt A_1(\lambda) := \mt F_0 - 2\lambda\delta \vc e_1\otimes \vc e_2,\qquad \mt A_2(\lambda) :=\mt F_0 - 2 \lambda \delta\vc v\otimes \vc e_1,\qquad\vc v:=\frac1{1+\delta^2}(\delta \vc e_1 + \vc e_2).
$$
We have the following lemma:

\begin{lem}
\label{lemma1}
For any $\lambda\in[0,1]$ there exist $\mt Q_1,\mt Q_2\in SO(2)$, $\vc u_1,\vc u_2,\vc w_1,\vc w_2\in\R^2$ such that 
\beq
\label{R1lam}
\mt Q_i \mt A_i(1-\lambda) = \mt A_i(\lambda) + \vc w_i\otimes\vc u_i,\qquad i=1,2.
\eeq
Furthermore, 
\begin{align}
\label{1stR1Conn}
\vc w_1 &= \gamma\bigl(\delta(1-2\lambda)\vc e_1 + \vc e_2 \bigr), \qquad &&\vc u_1 = \vc e_1 , \\
\label{2ndR1Conn}
\vc w_2 &= \gamma\bigl(((1-2\lambda)\delta^2 +1)\vc e_1 - 2\lambda \delta \vc e_2 \bigr), \qquad &&\vc u_2 = \vc e_2,
\end{align}
with
\beq
\label{DefGamma}
\gamma:= 	\frac{2 \delta (2\lambda -1)}{1+ \delta^2 (1-2\lambda )^2}
\eeq
and
\begin{align}
\begin{split}
\mt Q_1 &=  \frac{1}{1+\delta^2(1-2\lambda)^2}\left[\begin{array}{ ccc } 1-\delta^2(1-2\lambda)^2 & 2\delta (1-2\lambda) \\ -2\delta (1-2\lambda)  & 1-\delta^2(1-2\lambda)^2\end{array}\right], \\
\mt Q_2 & = \frac{1}{1+\delta^2(1-2\lambda)^2}\left[\begin{array}{ ccc } 1-\delta^2(1-2\lambda)^2 & -2\delta (1-2\lambda) \\ 2\delta (1-2\lambda)  & 1-\delta^2(1-2\lambda)^2\end{array}\right] .
\end{split}
\end{align}
\end{lem}

\begin{proof}
We first give the argument for $i=1$: In this case we note that the matrices $\mt A_1(\lambda)$ and $\mt A_1(1-\lambda)$ are of the same form as $\mt F_0$, $\mt F_0^{-1}$, where however $\delta$ is replaced by $\tilde{\delta}=(1-2\lambda)\delta$. As a consequence, we have the same rank-one connection as in the second identity in \eqref{Rk10} now with $\delta$ replaced by $\tilde{\delta}$, i.e.,
\begin{align}
\label{eq:case1}
\mt R_1 \mt A_1(1-\lambda) = \mt A_1(\lambda) + \frac{2\tilde{\delta}}{ 1 + \tilde{\delta}^2} (\tilde{\delta} \vc e_1 + \vc e_2) \otimes \vc e_1.
\end{align}
Here $\mt R_1 = \mt R_1(\tilde{\delta}) = \frac{1}{1+\tilde{\delta}^2}\left[\begin{array}{ ccc } 1-\tilde{\delta}^2 & 2\tilde{\delta} \\ -2\tilde{\delta} & 1-\tilde{\delta}^2  \end{array}\right]$. Rewriting \eqref{eq:case1} in terms of $\delta$ implies the claim.

For $i=2$ we make use of a different approach. We argue by first deducing necessary conditions for the presence of a rank-one connection and then show that these are also sufficient. Let us first notice that, if there exists $\mt Q\in SO(2), \vc u,\vc w\in\R^2$ such that \eqref{R1lam} is satisfied, then there exist $\vc u^\perp \perp \vc u$ such that $|\mt A_2(\lambda)\vc u^\perp|=|\mt A_2(1-\lambda)\vc u^\perp|$. Some computations imply that $\vc u^\perp$ is either parallel to $\vc e_1$ or to $\vc e_2$. We can hence choose $\vc u_2 $ either equal to $\vc e_1$ or to $\vc e_2$. The first case is trivial, and gives $\mt Q_2=\mt 1$. We hence focus on the case $\vc u_2 = \vc e_2.$ Since $\det \mt A_2(\lambda)=\det \mt A_2(1-\lambda)=1$, by taking the determinant of \eqref{R1lam} we have
$$
1 = \det \bigl(\mt Q_2 \mt A_2(1-\lambda)\bigr) = \det \bigl(\mt A_2(\lambda) + \vc w_2 \otimes\vc u_2 \bigr) = \det \bigl(\mt A_2(\lambda)\bigr) \bigl(1 +\mt A_2^{-1}(\lambda)\vc w_2\cdot \vc u_2 \bigr),
$$
and therefore, $ \mt A_2(\lambda)^{-1}\vc w_2\cdot \vc u_2=0.$ As a consequence $\vc w_2$ must be as in \eqref{2ndR1Conn}. As a consequence, if we can find $\gamma_2$ such that
\beq
\label{equazGT}
\mt A_2^T(1-\lambda)\mt A_2(1-\lambda) = (\mt A_2(\lambda) + \gamma\vc w_2\otimes \vc u_2)^T(\mt A_2(\lambda) + \gamma\vc w_2\otimes \vc u_2),
\eeq
polar decomposition gives \eqref{R1lam}. 
A calculation shows that this is indeed solvable with $\gamma$ as claimed.
\end{proof}

Let us now define
\begin{align}
\label{eq:F_i}
\begin{split}
\mt F_i (\mu,\lambda)&:= \mt A_i(\lambda) + \mu  \vc w_i(\lambda)\otimes \vc u_i	, \\ 
\mt C_i (\mu,\lambda)&:= \bigl(\mt A_i(\lambda) + \mu  \vc w_i(\lambda)\otimes \vc u_i \bigr)^T(\mt A_i(\lambda) + \mu \vc w_i(\lambda)\otimes \vc u_i ),
\end{split}
\end{align}
for $i=1,2$ (see Figure \ref{fig:fig3}).
Spelling this out, we have
\begin{align}
\label{C1Mat}
&{\mt C}_1(\mu,\lambda) = \left[\begin{array}{ ccc } 1 + \frac{4\delta^2 (2\lambda - 1)^2 (\mu^2 - \mu)}{\delta^2(2\lambda - 1)^2 + 1} & \delta(1 -2\lambda)( 1-2\mu) \\ \delta(1 -2\lambda)(1 - 2\mu 
) & 1 + (\delta - 2\delta\lambda)^2   \end{array}\right],
\\
\label{C2Mat}
&{\mt C}_2(\mu,\lambda) = \left[\begin{array}{ ccc } 1 +\frac{4\delta^2(\lambda^2 - \lambda)}{\delta^2 + 1} & \delta(1-2 \lambda) (1-2\mu ) \\ \delta(1-2 \lambda) (1-2\mu ) & 1+\delta^2+\frac{4\delta^2(\delta^2 + 1)(2\lambda - 1)^2(\mu^2 - \mu)}{\delta^2(2\lambda - 1)^2 + 1}  \end{array}\right].
\end{align}

\begin{rmk}[Symmetries]
\label{rmk:symm}
We remark that the expressions for $\mt C_1(\mu,\lambda)$, $\mt C_2(\mu, \lambda)$ enjoy good symmetries with respect to switching $\mu \mapsto 1-\mu$: For $j\in \{1,2\}$
\begin{align*}
&\mt C_{j}(\mu,\lambda) \vc e_j \cdot \vc e_j = \mt C_{j}(1-\mu, \lambda) \vc e_j \cdot \vc e_j.
\end{align*}

For fixed $j\in \{1,2\}$ we think of $\mt C_{j}(\mu,\lambda)$ and in particular of $\mu, \lambda$ as coordinates on the manifold $\mathcal{G}$. Each coordinate representation $\mt C_j(\mu, \lambda)$ (or $\mt F_j(\mu, \lambda)$) is particularly well-adapted to one of the rank-one connections from Lemma \ref{lemma1}.  Further, as we shall see in Lemma \ref{lemma2} and in Lemma \ref{lemma3} (see also Figure \ref{fig:inapprox} below) these coordinates are ``orthogonal'' in the Cauchy-Green space. Indeed, by varying $\mu$ in $\mt C_1$ (resp. $\mt C_2$) we can change $c_{11}$ (resp. $c_{22}$) by keeping $c_{22}$ (resp. $c_{11}$) constant.
\end{rmk}

Using the representations \eqref{C1Mat}, \eqref{C2Mat} and Proposition \ref{prop:Kqc}, it is immediate to verify that
\begin{align*}
\mathcal{G} 
&= \bigl\{\mt C \in\R^{2\times2}_{Sym^+}: \mt C = \mt C_1(\mu,\lambda),\, \mu,\lambda\in[0,1] 	\bigr\}\\
& = \bigl\{\mt C \in\R^{2\times2}_{Sym^+}: \mt C = \mt C_2(\mu,\lambda),\, \mu,\lambda\in[0,1] 	\bigr\}.
\end{align*}
As a consequence, for any $\mt F\in K^{qc}$ there exists $\mt R_1,\mt R_2\in SO(2)$ and $(\mu_1,\lambda_1),(\mu_2,\lambda_2)\in[0,1]^2$ such that 
$$
\mt F = \mt R_1 \mt F_1(\mu_1,\lambda_1) = \mt R_2 \mt F_2(\mu_2,\lambda_2).
$$
Thanks to Proposition \ref{prop:Kqc} and the explicit form of the Cauchy-Green tensors \eqref{C1Mat}, \eqref{C2Mat}, $\mt F\in \inter K^{qc}$ if and only if $(\mu_1,\lambda_1),(\mu_2,\lambda_2)\in(0,1)^2$, $\lambda_i\neq \frac12.$ 

\begin{figure}
\begin{subfigure}{.45\textwidth}
  \centering
  \includegraphics[width=0.9\linewidth,page=5]{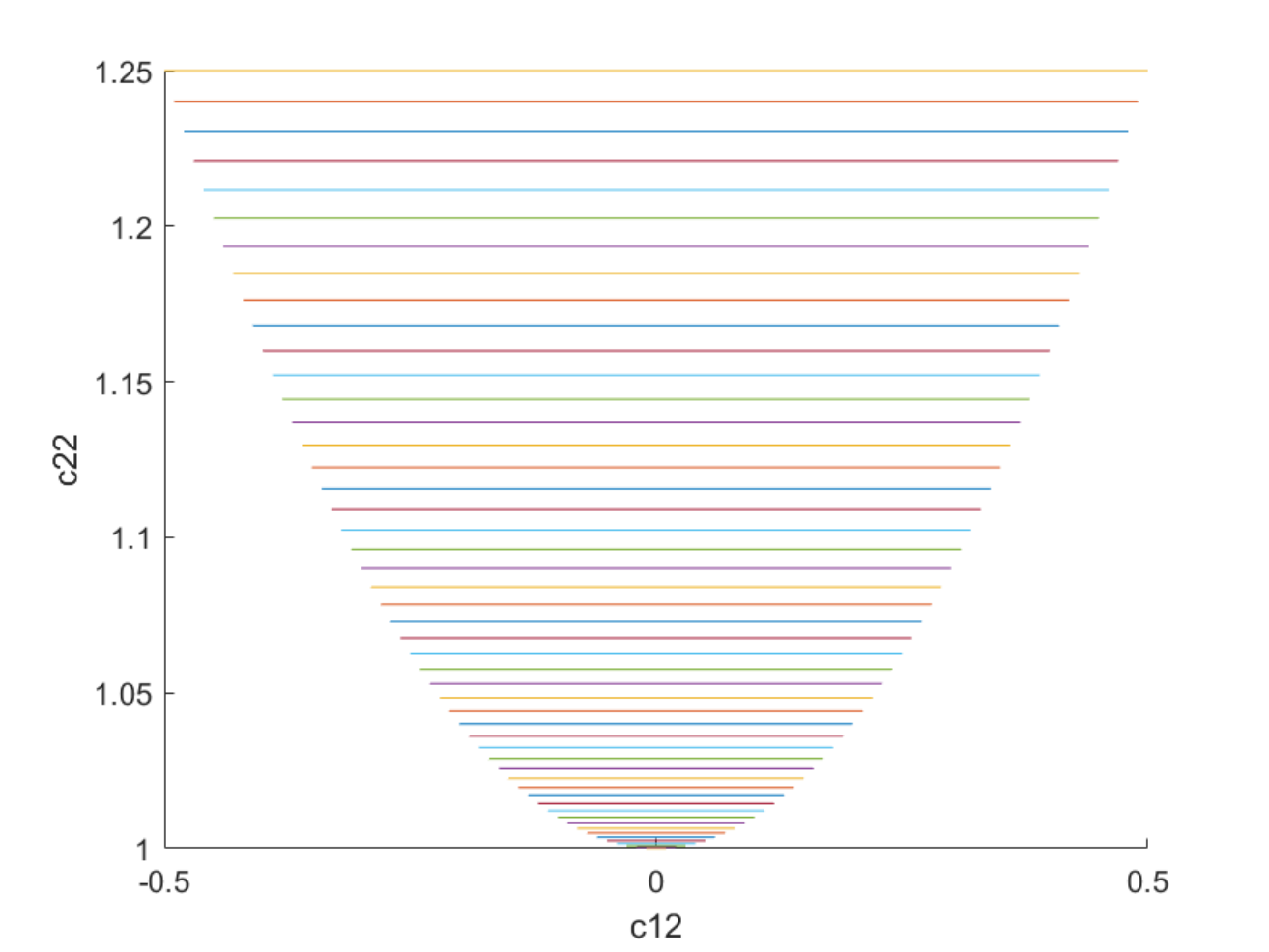}
  \caption{}
  \label{fig99}
\end{subfigure}%
\begin{subfigure}{.45\textwidth}
  \centering
  \includegraphics[width=0.9\linewidth,page=6]{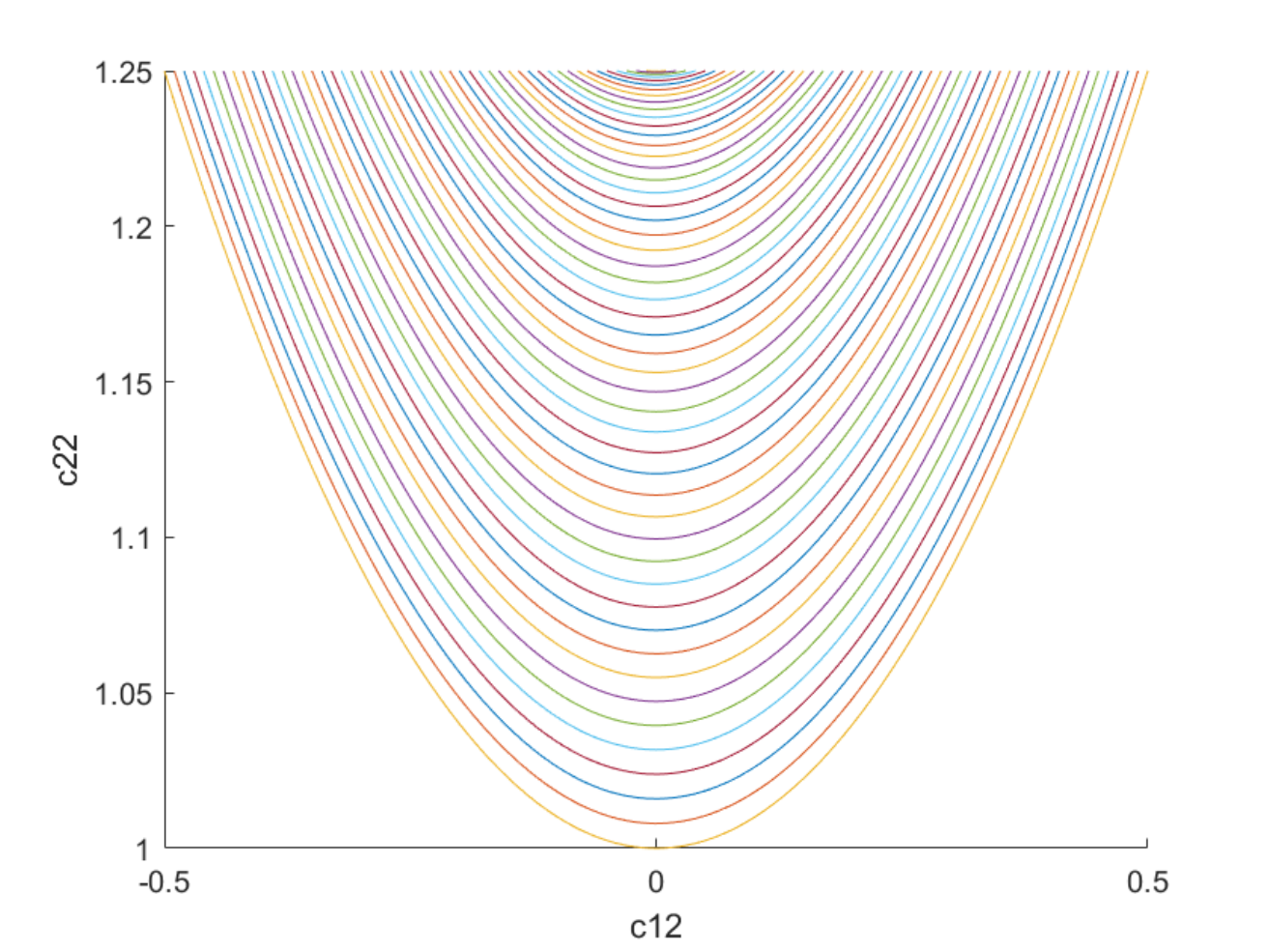}
    \caption{}
  \label{fig1}
\end{subfigure}%
\caption{Plotting $c_{22}$ versus $c_{12}$ in case $\delta=0.5$ when $\mu$ varies in $[0,1]$. On the left for $i=1$, on the right for $i=2.$}
\label{fig:fig3}
\end{figure}

\begin{lem}
\label{lemma2}
Let $\mt F = \mt R\mt F_1(\mu,\lambda) \in \inte( \mathcal{G})$ for some $(\mu,\lambda)\in(0,1)^2,$ $\lambda\neq \frac12$ and some $\mt R\in SO(2)$. Let further $\eps_0:= 1 - \mt C_1(\mu,\lambda)\vc e_1\cdot\vc e_1>0$. Then, for any $\eps\in(0,\eps_0]$, there exists $\mu^*,\rho\in[0,1]$ such that 
\begin{align}
\label{eq:decompose}
\begin{split}
&\mt F = \rho\mt R\mt F_1(\mu^*,\lambda)+(1-\rho) \mt R\mt F_1(1-\mu^*,\lambda),\\ 
& \rank(\mt F_1(\mu^{\ast},\lambda)- \mt F_1(1-\mu^{\ast},\lambda))=1,\\
&\mt C_1(\mu^*,\lambda)\vc e_1\cdot\vc e_1 = \mt C_1(1-\mu^*,\lambda)\vc e_1\cdot\vc e_1 = 1-\eps.
\end{split}
\end{align}
In particular, $\mt C_1(\mu^{\ast},\lambda) \vc e_2 \cdot \vc e_2 = \mt C_1(\mu,\lambda) \vc e_2 \cdot \vc e_2$.

Furthermore, if $\lambda(1-\lambda) \in (0,\frac{3}{16})$ and $\eps_0 \in (0, 2^{-4} \min\{ \delta^2,1\})$, there exists a constant $C_1=C_1(\delta)\geq 1$ such that for $\chi = \sgn\bigl(\bigl(\frac12 -\mu\bigr)\bigl(\frac12 -\mu^*\bigr)\bigr)$
\begin{align}
\label{eq:control}
\left|\rho - \frac{1}{2}(\chi+1) \right| \leq C_1 \eps_0,
\end{align}
and
\begin{align}
\label{eq:distF}
\left|\mt F - \mt R \mt F_1\left(\mu^{\ast} + \frac{1}{2}(1-\chi)(1-2\mu^{\ast}),\lambda\right) \right| \leq C_1 \eps_0.
\end{align}
\end{lem}

\begin{rmk}
\label{rmk:Lip_dep}
We remark that the Lipschitz dependence of  \eqref{eq:control} and \eqref{eq:distF} in terms of $\eps_0$ plays a crucial role for our argument in the sequel. It will allow us to make use of ``self-improving'' error estimates (see also the remarks following Lemma \ref{lem:Conti} and the argument for Lemmas \ref{lem:diamond}, \ref{lem:A4}). It is this ``self-improvement'' which in the sequel allows us to obtain exponential instead of super-exponential $BV$ estimates.
\end{rmk}

\begin{proof}
We first prove that for $\beta_1 := \frac{\eps}{1 - \mt C_1(\mu,\lambda)\vc e_1\cdot\vc e_1},$ we have
\beq
\label{muStar}
\mu^* = \frac12\Bigl(1 \pm 	\sqrt{1 - 4\beta_1(\mu-\mu^2) }\Bigr),
\qquad 
\rho = \frac12\Bigl ( 1 +\chi \frac{\sqrt{1 - 4(\mu-\mu^2)}}{\sqrt{1 - 4\beta_1(\mu-\mu^2) } } 	\Bigr) ,
\eeq
where $\chi = \sgn \bigl(\bigl(\frac12 -\mu\bigr)\bigl(\frac12 -\mu^*\bigr)\bigr).$
To this end, we note that
\begin{align}
\label{eq:boundary}
\mt C_1(\mu,\lambda)\vc e_1\cdot\vc e_1 =1 + \frac{4\delta^2 (2\lambda - 1)^2 (\mu^2 - \mu)}{\delta^2(2\lambda - 1)^2 + 1} .
\end{align}
Solving the equation $1-\mt C_1(\mu^*,\lambda)\vc e_1\cdot\vc e_1 = \eps$, we obtain the two solutions 
$$
\mu^*_{\pm } = \frac12\Biggl(1 \pm 	\sqrt{1-\eps\frac{1+\delta^2(2\lambda-1)^2}{\delta^2(2\lambda-1)^2}}\Biggr),
$$ 
which can be rewritten as in \eqref{muStar}. Then, using the definitions \eqref{eq:F_i} and the expression for $\mu^{\ast}$, we solve the identity
\begin{align*}
\mt F &= \mt R \mt F_1(\mu, \lambda) = \rho\mt R\mt F_1(\mu^*,\lambda)+(1-\rho) \mt R\mt F_1(1-\mu^*,\lambda)
\end{align*}
for $\rho$. This leads to $\rho = 1+\frac{\mu - \mu^*}{2\mu^*-1}$ and hence the desired decomposition \eqref{eq:decompose} holds. The fact that $\mt C_1(\mu^{\ast},\lambda ) \vc e_2 \cdot \vc e_2= \mt C_1(\mu,\lambda) \vc e_2 \cdot \vc e_2 $ is a direct consequence of the form of the Cauchy-Green tensor in \eqref{C1Mat}. In order to infer the second identity in \eqref{muStar}, it just remains to prove that $\rho$ can be written as in \eqref{muStar}. But this is true as $\mu$ can be written as $$
\mu = \frac12\Biggl(1 \pm 	\sqrt{1-\eps_0\frac{1+\delta^2(2\lambda-1)^2}{\delta^2(2\lambda-1)^2}}\Biggr),
$$
where $1-\mt C_1(\mu,\lambda)\vc e_1\cdot\vc e_1 = \eps_0$. Therefore,
$$
\rho = \frac12 + \chi \frac12 \frac{\sqrt{1-\eps_0\frac{1+\delta^2(2\lambda-1)^2}{\delta^2(2\lambda-1)^2}}}{\sqrt{1-\eps\frac{1+\delta^2(2\lambda-1)^2}{\delta^2(2\lambda-1)^2}}}	= \frac12\Bigl ( 1 +\chi \frac{\sqrt{1 - 4(\mu-\mu^2)}}{\sqrt{1 - 4\beta_1(\mu-\mu^2) } } 	\Bigr).
$$

In order to deduce \eqref{eq:control}, it therefore remains to estimate the quantities in \eqref{muStar}.
Exploiting \eqref{eq:boundary} in combination with the assumption that $\lambda \in (0,\frac{1}{4})\cup (\frac{3}{4},1)$, we obtain
\begin{align*}
|\mu (\mu-1)| \leq \frac65 (1+\delta^{-2}) \eps_0 =: \tilde{C}_1 \eps_0
\end{align*}
for some constant $\tilde{C}_1=\tilde{C}_1(\delta)>1$.
Assuming without loss of generality, that $\chi = 1$ and Taylor expanding the square root, we then deduce that
\begin{align*}
|\rho- 1| = \frac{1}{2} \left| \sqrt{ \frac{1-4(\mu-\mu^2)}{1-4\beta_1 (\mu - \mu^2)}} -1 \right|
\leq (1+\beta_1)|\mu-\mu^2| + O(|\mu-\mu^2|^2)
\leq C_1 \eps_0,
\end{align*}
with $C_1>0$ depending only on $\delta$ (as $\beta_1 \leq 1$).
Arguing similarly in the cases that $\chi = -1$ and
noting that $\mu = \frac{1}{2}= \mu^{\ast}$ is impossible by the above smallness assumptions hence concludes the proof. 
\end{proof}

In a similar way we also obtain an analogous splitting result in the second rank-one system:

\begin{lem}
\label{lemma3}
Let $\mt F = \mt R\mt F_2(\mu,\lambda)\in \inte(\mathcal{G})$ for some $(\mu,\lambda)\in(0,1)^2,$ $\lambda\neq \frac12$ and some $\mt R\in SO(2)$. Let $\eps_0:= 1 +\delta^2 - \mt C_2(\mu,\lambda)\vc e_2\cdot\vc e_2$. Then, for any $\eps\in (0,\eps_0]$, there exists $\mu^*,\rho\in[0,1]$ such that
\begin{align}
\label{eq:decompose2}
\begin{split}
&\mt F = \rho\mt R\mt F_2(\mu^*,\lambda)+(1-\rho) \mt R\mt F_2(1-\mu^*,\lambda),\\ 
& \rank(\mt F_2(\mu^{\ast},\lambda)- \mt F_2(1-\mu^{\ast},\lambda))=1,\\
&\mt C_2(\mu^*,\lambda)\vc e_2\cdot\vc e_2 = \mt C_2(1-\mu^*,\lambda)\vc e_2\cdot\vc e_2 = 1-\eps.
\end{split}
\end{align}
In particular, $\mt C_2(\mu^{\ast},\lambda) \vc e_1 \cdot \vc e_1 = \mt C_2(\mu,\lambda) \vc e_1 \cdot \vc e_1$.

Furthermore, if $\lambda(1-\lambda) \in (0,\frac{3}{16})$ and $\epsilon_1 \in (0, 2^{-4} \min\{ \delta^2,1\})$, there exists a constant $C_1=C_1(\delta)\geq 1$ such that for $\chi = \sgn\bigl(\bigl(\frac12 -\mu\bigr)\bigl(\frac12 -\mu^*\bigr)\bigr)$
\begin{align}
\label{eq:control2}
\left|\rho - \frac{1}{2}(\chi+1) \right| \leq C_1 \eps_0,
\end{align}
and
\begin{align}
\label{eq:distF2}
\left|\mt F - \mt R \mt F_2\left(\mu^{\ast} + \frac{1}{2}(1-\chi)(1-2\mu^{\ast}),\lambda \right) \right| \leq C_1 \eps_0.
\end{align}
\end{lem}

\begin{proof}
The argument proceeds along the same lines as the proof of Lemma \ref{lemma2}. We hence omit the details.
\end{proof}

\section{In-Approximation}
\label{sec:in_approx}
In this section, we discuss the in-approximation which we will be using in the sequel (see also Figure \ref{fig:inapprox}).

\begin{defi}
\label{defi:in_approx}
Let $k\in \N$, $k\geq 1$, let $\mt F \in  \inte(K^{qc})$ and denote by $\mt C(\mt F) = \mt F^T \mt F$ the associated Cauchy-Green tensor. We set $\zeta_0:=2^{-4}\min\{\delta^{2},1\}$, and
\begin{align*}
\mathcal{V}_{2k} &:= 
\Set{\mt F \in \inter K^{qc} \,\Bigg|\; \text{\parbox{2.8in}{\centering 
 $\left|\mt C(\mt F) \vc e_1 \cdot \vc e_1 -1\right|\in \zeta_0{2^{-(k+3)}}\cdot (5,7),$ $\left|\mt C(\mt F) \vc e_2 \cdot \vc e_2 -(1+\delta^2)\right|\in \zeta_0 2^{-k} \cdot(1,2)$
}}},\\
\mathcal{V}_{2k+1} &:= 
\Set{\mt F \in \inter K^{qc} \,\Bigg|\; \text{\parbox{2.8in}{\centering 
 $\left|\mt C(\mt F) \vc e_1 \cdot \vc e_1 - 1\right|\in \zeta_0{2^{-(k+1)}}\cdot (1,2),$ $\left|\mt C(\mt F) \vc e_2 \cdot \vc e_2 -(1+\delta^2)\right|\in \zeta_0 2^{-(k+3)} \cdot(5,7)$
 }}}.
\end{align*}
Further we define
\begin{align*}
\mathcal{V}_{1} &:= 
\Set{\mt F \in \inter K^{qc}\setminus\left(\bigcup\limits_{k=2}^{\infty} \mathcal{V}_{k}\right) \,\Bigg|\; \text{\parbox{2.8in}{\centering 
$\left|\mathsf{C}(\mathsf{F}) \mathbf{e}_1 \cdot \mathbf{e}_1 - 1 \right|\in 2^{-(m+1)}\zeta_0 \cdot (1,2)$ for some $m\geq 1$ and $\left|\mt C (\mt F) \vc e_2 \cdot \vc e_2-(1+\delta^2)\right| \geq 7\cdot 2^{-(m+3)}\zeta_0$
}}},\\
\mathcal{V}_{0} &:= \inte(K^{qc}) \setminus\left(\bigcup\limits_{k=0}^{\infty} \mathcal{V}_{k}\right).
\end{align*}
\end{defi}

\begin{figure}
\centering
\begin{tikzpicture}[scale=4]

\draw[->,thick,black] ({0.4*2.001},0.8) -- ({0.4*2.001},2.1);
\draw[->,thick,black] ({0.4*2.001},0.8) -- ({1.1*2.001},0.8);

\filldraw [black] ({1.1*2.001},0.8) circle (0pt) node[anchor=north,black] {$\mt C\vc e_1\cdot\vc e_1$};
\filldraw [black] ({0.4*2.001},2.1) circle (0pt) node[anchor=west,black] {$\mt C\vc e_2\cdot\vc e_2$};

\draw[-,black] ({(0.4-0.01)*2.001},2) -- ({(0.4+0.01)*2.001},2);\filldraw [black] ({(0.4-0.01)*2.001},2) circle (0pt) node[anchor=east,black] {$1+\delta^2$};
\draw[-,black] ({(0.4-0.01)*2.001},1) -- ({(0.4+0.01)*2.001},1);\filldraw [black] ({(0.4-0.01)*2.001},1) circle (0pt) node[anchor=east,black] {$1$};
\draw[-,black] ({0.5*2.001},0.8-0.02) -- ({0.5*2.001},0.8+0.02);\filldraw [black] ({0.5*2.001},0.8-0.02) circle (0pt) node[anchor=north,black] {$\frac{1}{1+\delta^2}$};
\draw[-,black] ({1*2.001},0.8-0.02) -- ({1*2.001},0.8+0.02);\filldraw [black] ({1*2.001},0.8-0.02) circle (0pt) node[anchor=north,black] {$1$};

\filldraw [black] ({2.001*0.8},2) circle (0pt) node[anchor=north,black] {$\dots$};
\filldraw [black] ({2.001*1.005},1.65) circle (0pt) node[anchor=south,black,rotate=90] {$\dots$};

\fill [cyan, domain=  { (2-2^(-4))^(-1) }:{1-2^(-5)}, variable=\y] 
({(1-2^(-5))*2.001},{2-2^(-4)}) -- 
plot ({\y*2.001}, {1/\y})
 -- cycle;

\foreach \x in {2,3,4}
\fill [green, domain=  {((1  - 2^(-\x))}:{(1  - 2^(-\x-1))}, variable=\y] 
({(1  - 2^(-\x))*2.001},{2 - 2^(-\x) + 0.25*2^(-\x-1) }) -- 
plot ({\y*2.001}, {1/\y})
--  ({(1  - 2^(-\x-1))*2.001},{2 - 2^(-\x)+ 0.25*2^(-\x-1)}) -- cycle;

\foreach \x in {2,3,4,5}
\draw[black,fill=red] ({(1  - 7*2^(-\x-3))*2.001},{2 - 2^(-\x+1)}) -- ({(1  - 5*2^(-\x-3))*2.001},{2 - 2^(-\x+1)}) -- ({(1  - 5*2^(-\x-3))*2.001},{2 - 2^(-\x)}) --  ({(1  - 7*2^(-\x-3))*2.001},{2 - 2^(-\x)}) -- cycle;

\foreach \x in {2,3,4,5}
\draw[black,fill=yellow] ({(1  - 2^(-\x))*2.001},{2 - 7*2^(-\x-3)}) -- ({(1  - 2^(-\x-1))*2.001},{2 - 7*2^(-\x-3)}) -- ({(1  - 2^(-\x-1))*2.001},{2 - 5*2^(-\x-3)}) --  ({(1  - 2^(-\x))*2.001},{2 - 5*2^(-\x-3)}) -- cycle;

\foreach \x in {1,2,3,4}
\draw[-,black] ({(1)*2.001},{2 - 2^(-\x)}) -- ({((2 - 2^(-\x))^(-1) )*2.001},{2 - 2^(-\x)});

\foreach \x in {2,3,4,5}
\draw[-,black] ({(1  - 2^(-\x))*2.001},{2}) -- ({(1 - 2^(-\x))*2.001},{ (1 - 2^(-\x))^(-1) });

\draw[-,thick,black] ({0.5*2.001},2) -- ({1*2.001},2);
\draw[-,thick,black] ({1*2.001},1) -- ({1*2.001},2);
\draw[domain=0.5:1,smooth,variable=\x,thick] plot ({2.001*\x},{1/\x});

\end{tikzpicture}
\caption{
\label{fig:inapprox}
Schematic representation of the sets $\mathcal{V}_k$ projected onto the $\mt C\vc e_1\cdot\vc e_1$, $\mt C\vc e_2\cdot\vc e_2$ plane. In this figure, the set $K$ coincides with the point $(1,1+\delta^2)$, while the set $\mathcal{G}$ (corresponding to $K^{qc}$) coincides with the region $\left\{(x,y)\in \left((1+\delta^2)^{-1},1 \right)\times \left(1,(1+\delta^2)\right):  y\geq x^{-1}\right\}. $
In red and yellow the sets $\mathcal{V}_k$ for $k\geq 2$ respectively when $k$ is even and odd. The set $\mathcal{V}_{0}$ coincides with the cyan region, plus the vertical lines in the green region, and the upper boundaries of the yellow boxes. The set $\mathcal{V}_1$ coincides with the green region, plus the lower boundary of the yellow boxes and the boundaries of the red boxes.}
\end{figure}
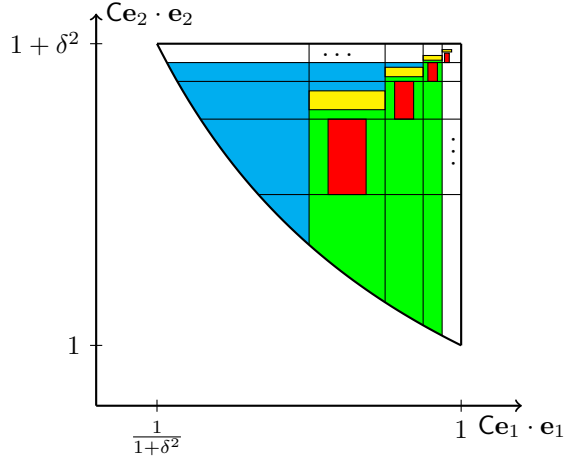

We claim that a suitable subset of these sets form an in-approximation for the two-dimensional geometrically non-linear two-well problem.
We start by giving the definition of in-approximation which we will be using in the sequel:

\begin{defi}
Let $K$ be as in \eqref{defK}. We say that a sequence of relatively open sets $\mathcal{W}_k\subset K^{qc},$ $k\geq 1$ is an in-approximation of $K$ in $K^{qc}$ if:
\begin{itemize}
\item $\mathcal{W}_k\subset\mathcal{W}_{k+1}^{lc}$ for each $k\geq 1$,
\item $\lim_{k\to \infty}\sup_{\mt X\in \mathcal{W}_k}\dist(\mt X, K)=0$.
\end{itemize}
\end{defi}

\begin{rmk}
Let $\mt F\in \inter K^{qc},$ and let $\mt G\in B(\mt F,\eps)\cap \{\mt H\colon \det \mt H=1\}$. Then, as the set $\{\mt H \in \R^{2\times 2}: \ \det(\mt H)=1\}$ forms a co-dimension one manifold which locally coincides with $K^{qc}$, for any $\eps>0$ small enough, $\mt G\in K^{qc}$. For this reason, a sequence of sets relatively open in $K^{qc}$ is also relatively open in the set $\R^{2\times2}\cap \{\mt H\colon \det \mt H=1\}$.
\end{rmk}

The following result holds:

\begin{lem}
\label{lem:in_approx}
Let $k\in \N,$ $k\geq 2$, and let the sets $\mathcal{V}_k$ be defined as in Definition \ref{defi:in_approx}. Then they form an in-approximation for the two-well problem with $K$ as in \eqref{defK}.
\end{lem}

\begin{proof}
Let $\{\mt F_k\}_{k\in\N}$ be a sequence such that $\mt F_k \in \mathcal{V}_k$.
Then, by definition, $\mt C(\mt F_k)$ is such that $|1-\mt C(\mt F_k) \vc e_1 \cdot \vc e_1| \leq C 2^{-\frac k2 }$, $|\mt C(\mt F) \vc e_2 \cdot \vc e_2 -( 1 + \delta^2)| \leq C 2^{-\frac k2 }$, for some $C>0$, and $\det(\mt F_k)=1$. As a consequence, $\lim_{k\to \infty}\sup_{\mt X\in \mathcal{V}_k}\dist(\mt X,K)=0$.

Hence it remains to prove that $\mathcal{V}_k \subset (\mathcal{V}_{k+1})^{lc}$. We argue in the case $k$ even, as the case $k$ odd can be treated similarly. Then, if $\mt F \in \mathcal{V}_{k}= \mathcal{V}_{2k_0}$ for some $k_0 \in \N$, we apply Lemma \ref{lemma3} to obtain a decomposition
\begin{equation}
\label{Decomp}
\mt F = \rho \mt R \mt F_2(\mu^{\ast},\lambda) + (1-\rho) \mt R \mt F_2(1-\mu^{\ast},\lambda),
\end{equation}
such that 
\begin{align*}
1+\delta^2 - \mt C_2(\mu^{\ast},\lambda) \vc e_2 \cdot \vc e_2 =  \varepsilon \in 2^{-(k_0+3)}\zeta_0 \cdot (5,7).
\end{align*}
By symmetry (see Remark \ref{rmk:symm}) the same also holds for $\mt{C}_2(1-\mu^{\ast},\lambda) \vc e_2 \cdot \vc e_2$ and as $\mt C_2(\mu^{\ast},\lambda) \vc e_1 \cdot \vc e_1 = \mt{C}_2(1-\mu^{\ast},\lambda) \vc e_1 \cdot \vc e_1 = \mt{C}(\mt F) \vc e_1 \cdot \vc e_1$, we infer that $\mt R \mt F_2(\mu^{\ast},\lambda)$, $\mt R \mt F_2(1-\mu^{\ast},\lambda) \in \mathcal{V}_{k+1}$ and thus $\mt  F \in (\mathcal{V}_{k+1})^{lc}$ which proves the claim.
\end{proof}

\begin{rmk}
Let us comment on our choice of in-approximation: Our in\hyp approximation is designed in such a way that for $k\geq 2$ the sets $\mathcal{V}_{k}$ approximate the energy wells (in Cauchy-Green space) dyadically. We seek to iteratively move from $\mathcal{V}_k$ to $\mathcal{V}_{k+1}$ (when $k\geq2$) by splitting along rank-one segments. 

The sets $\mathcal{V}_0$ and $\mathcal{V}_1$ are less restrictive than the sets $\mathcal{V}_k$ with $k\geq 2$, in that any datum in $\mathcal{G}$ which is \emph{not} already in a dyadic neighbourhood of the wells is contained in these. However,
the full sequence $\{\mathcal{V}_k\}_{k\geq 0}$ is not an in-approximation. Indeed, for $k\in \{0,1\}$, the inclusion $\mathcal{V}_k\subset\mathcal{V}_{k+1}^{lc}$ does not hold in general. 
Nonetheless, using the splitting results from Lemmas \ref{lemma2}, \ref{lemma3}, any $\mt F\in \inter K^{qc}$ can be moved within two replacement constructions from $\mathcal{V}_{0},\mathcal{V}_{1}$ to some (not necessarily the first) dyadic neighbourhood $\mathcal{V}_k$, $k \geq 2$, of the wells (see Lemmas \ref{lem:diamond_replace}, \ref{lem:A3}).
\end{rmk}

\section{The Replacement Construction in Suitable Diamond-Shaped Domains}

In this section we recall the non-linear replacement construction from \cite{C}. Together with the matrix space geometry discussed in the previous section, this will serve as the building block for our replacement construction and will allow us to verify the conditions (A2), (A3), (A4) from \cite{RZZ18} (which will be recalled and explained in Section \ref{sec:Covering} below).

\begin{lem}[Lemma 2.3 in \cite{C}]
\label{lem:Conti}
Let $\lambda \in (0,1)$ and $h \in(0,8^{-1})$.
Let $\mt A, \mt B, \mt C \in \R^{2\times 2}$ with $\mt C=\lambda \mt A + (1-\lambda) \mt B$, $\det \mt A = \det \mt B = \det \mt C = 1,$ and $\rank(\mt A -\mt B)\leq 1$. 
Then there exists a domain $\Omega_{h,\mt R}:= \mt R \conv(\{ \pm h \vc e_{1}, \pm \vc e_2\})$ with $\mt R \in SO(2)$, and a deformation $\vc u: \Omega_{h,\mt R} \rightarrow \R^2$ such that for some constant $C_1\geq 1$ which does not depend on $h$ or $\lambda$ 
\begin{itemize}
\item[(i)] $\vc u(\vc x)= \mt C \vc x$ on $\partial \Omega_{h, \mt R}$,
\item[(ii)] $\nabla \vc u$ attains five values and its level-sets are the union of (up to a set of measure zero) two disjoint open triangles, each of perimeter less than $\Omega_{h, \mt R},$
\item[(iii)] $\det(\nabla \vc u)=1$ a.e. in $\Omega_{h, \mt R}$,
\item[(iv)] $\min\{|\nabla \vc u - \mt A|, |\nabla\vc u -\mt B|\} \leq e:= C_1|\mt A-\mt B| \min\{\lambda, 1-\lambda\}h$ a.e. in $\Omega_{h, \mt R}$,
\item[(v)] $|\vc u(\vc x)-\mt C\vc x|\leq C_1|\mt A-\mt B| (1-\lambda)\lambda h$ for any $\vc x\in\Omega_{h,\mt R}$. 
\item[(vi)] Assuming that $\lambda \leq 1- \lambda$ we have
\begin{align*}
\Omega_{\mt A}:=\frac{|\{ \vc x \in \Omega_{h, \mt R}: \ |\nabla \vc u(\vc x) - \mt A|<|\nabla \vc u( \vc x) - \mt B|\}|}{|\Omega_{h, \mt R}|} &= \lambda (1+(1-\lambda)h),\\
\Omega_{\mt B}:=\frac{|\{\vc x \in \Omega_{h, \mt R}: \ |\nabla \vc u( \vc x) - \mt B|<|\nabla \vc u( \vc x) - \mt A|\}|}{|\Omega_{h, \mt R}|} &= 1-\lambda (1+(1-\lambda)h).
\end{align*}
If $\lambda > 1-\lambda$ the bounds in (iv) are reversed.
\end{itemize}
\end{lem}

Let us comment on this construction. For our purposes crucial properties of this replacement construction are:
\begin{itemize}
\item The preservation of the determinant constraint. This allows us to use this replacement construction iteratively without leaving the set $K^{qc}$. Precursor constructions of this can be found in \cite{MS} and \cite{CT05}. The construction from \cite{C}, however, is most convenient for us as all controls are explicit and only a small number of deformation gradients are involved.
\item In the properties (iv)-(v) there is a ``self-improvement'' of the errors in the construction as soon as $\mt C \in \mathcal{V}_k$ with $k\geq 2$. Indeed, in our application of the replacement construction (see Lemmas \ref{lem:diamond_replace}, \ref{lem:diamond}) for $\mt C \in \mathcal{V}_k$, $\mt A,\mt B\in \mathcal{V}_{k+1}$ we will need to ensure that 
\begin{align}
\label{eq:error_est}
e=C_1|\mt A - \mt B|\min\{\lambda, 1-\lambda\} h \sim 2^{-\frac{k}{2}} \zeta_0,
\end{align}
in order to ensure that the replacement function $\nabla \vc u\in \mathcal{V}_{k+1}$ a.e. in $\Omega_{h,\mt R}$.
As the splitting from Lemmas \ref{lemma2}, \ref{lemma3} allows us to choose $\lambda \sim 2^{-\frac{k}{2}}$, and since in our two-well setting $|\mt A - \mt B| \leq 2\delta (1+\delta)^2$, we are always able to achieve the desired error estimate in \eqref{eq:error_est} with a \emph{uniform} choice of $h$ (independent of $k$). We will exploit this heavily in deriving exponential (in contrast to super-exponential) $BV$ estimates.
\item  The volume fractions in the replacement construction in (vi) are ``self-improving''. Indeed, if as in the remark above we take $\mt C \in \mathcal{V}_k$, $\mt A,\mt B\in \mathcal{V}_{k+1}$ with $k\geq 2$ and $\lambda \sim 2^{-k}$ we obtain that the proportion of the domain which is not in the same well as $\mt C$ will also be proportional to $\lambda$. This plays a crucial role for our argument as it will yield geometrically improving $L^1$ estimates.
\end{itemize}

\begin{proof}
We only sketch the proof as the main difference with respect to the argument in \cite{C} is the quantitative dependence of the estimates on $h,\lambda$. In particular, properties (i)-(iii) do not require any modifications. We note that, since $\rank(\mt A-\mt B)\leq 1$, there exist $\vc a, \vc n$ such that $\mt A = \mt B + \vc a\otimes \vc n$. By subtracting $\mt C$, rescaling by $\rho = |\mt A-\mt B|$, rotating, and finally adding the identity matrix, we can replace $\mt A,\mt B$ by 
$$
\tilde{\mt A}= \mt{Id} + (1-\lambda) {\vc e}_2 \otimes {\vc e}_1,\qquad
\tilde{\mt B}= \mt{Id} - \lambda{\vc e}_2 \otimes \vc e_1,
$$ 
$\mt C$ by $\mt{Id}$ and $\vc u (\vc x)$ by $\tilde{\vc u} (\vc x) = \frac{1}{\rho}\mt Q ( {\vc u} (\mt R\vc x) -\mt C\mt R\vc x) +\vc x$ for some $\mt Q,\mt R\in SO(2)$ such that $\mt Q\vc a = \vc e_2$ and $\mt R^T\vc n = \vc e_1$. Further, without loss of generality $\lambda \leq 1-\lambda$. We first construct a deformation gradient $\nabla \vc v$ as in Figure \ref{ContiCostruction} (left), where
$$
\mt E= \mt{Id} - q(1-\mu) {\vc e}_1 \otimes{\vc e}_2,\qquad
\mt F= \mt{Id} + q \mu {\vc e}_1 \otimes {\vc e}_2,
$$ 
with $q = h^2 \frac{\lambda (1-\lambda)}{\mu (1-\mu)}$, as in \cite{C1}. This is possible by using that both $ \tilde{\mt A}$ and $\tilde{\mt B}$ as well as $\mt E$ and $\mt F$ are rank-one connected. The choice of $q$ is determined by the requirement of preserving the determinant. We note that at the points $(0,1)$, $(0,-1)$, $(h,0)$, $(-h,0)$ the function $\vc v$ can be chosen to be equal to the identity map. 

In order to match the boundary conditions in (i), we then define $\nabla \tilde{\vc u}$ to be equal to $\nabla \vc v$ in the orange, the green, and the blue regions of Figure \ref{ContiCostruction} (right). In the red and yellow regions we define $\tilde{\vc u}$ by affine interpolation of the values at the corner points (in particular, the desired boundary conditions are then satisfied on the whole boundary of the diamond). This leads to $\nabla \tilde{\vc u} = \mt G$ and $\nabla \tilde{\vc u} = \mt H$ respectively. Here, 
$$
\mt G= \mt{Id} - p_{\mt G}(-h{\vc e}_1+{\vc e}_2)\otimes ({\vc e}_1+h{\vc e}_2),\qquad
\mt H = \mt{Id} - p_{\mt H}(h{\vc e}_1+{\vc e}_2)\otimes ({\vc e}_1-h{\vc e}_2),
$$
where 
$$
p_{\mt G}=
 \frac{\lambda(1-\lambda)}{(1-\lambda)(1-h\lambda)-\mu},\qquad p_{\mt H}=
 \frac{\lambda(1-\lambda)}{(1-\lambda)(1+h\lambda)-\mu}.
$$

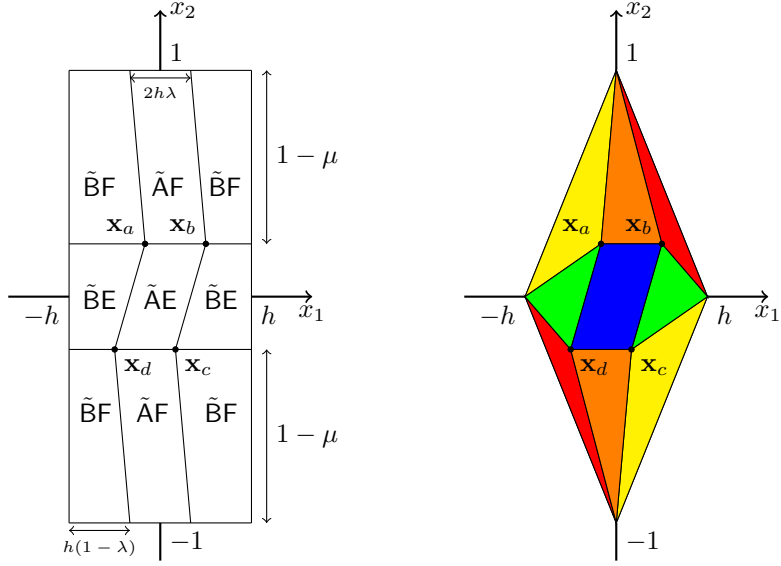
\begin{figure}
\centering
\begin{tikzpicture}
\draw[-,thick,black] (-2,0) -- (-1.2,0);
\draw[-,thick,black] (0,-3.5) -- (0,-3);
\draw[->,thick,black] (1.2,0) -- (2,0);
\draw[->,thick,black] (0,3) -- (0,3.8);

\filldraw [black] (2,0) circle (0pt) node[anchor=north,black] {$x_1$};
\filldraw [black] (0,3.8) circle (0pt) node[anchor=west,black] {$x_2$};

\filldraw [black] (1.2,0) circle (0pt) node[anchor=north west,black] {$h$};
\filldraw [black] (-1.2,0) circle (0pt) node[anchor=north east,black] {$-h$};
\filldraw [black] (0,3) circle (0pt) node[anchor=south west,black] {$1$};
\filldraw [black] (0,-3) circle (0pt) node[anchor=north west,black] {$-1$};

\draw[black,thin] (-1.2,0.7) -- (1.2,0.7);
\draw[black,thin] (-1.2,-0.7) -- (1.2,-0.7);

\draw[black,thin] ({-0.4+0.2},0.7) -- (-0.4,3);
\draw[black,thin] ({0.4+0.2},0.7) -- (0.4,3);
\draw[black,thin] ({-0.4-0.2},-0.7) -- (-0.4,-3);
\draw[black,thin] ({0.4-0.2},-0.7) -- (0.4,-3);
\draw[black,thin] ({-0.4+0.2},0.7) -- ({-0.4-0.2},-0.7);
\draw[black,thin] ({0.4+0.2},0.7) -- ({0.4-0.2},-0.7);

\draw[black,thin] (-1.2,-3) -- (-1.2,3);
\draw[black,thin] (1.2,-3) -- (1.2,3);
\draw[black,thin] (-1.2,3) -- (1.2,3);
\draw[black,thin] (-1.2,-3) -- (1.2,-3);

\filldraw [red] (-0.8,1.5) circle (0pt) node[black] {$\tilde{\mt B}\mt{F}$};
\filldraw [red] (-0.8,0) circle (0pt) node[black] {$\tilde{\mt B}\mt{E}$};
\filldraw [red] ({-0.8-0.05},-1.5) circle (0pt) node[black] {$\tilde{\mt B}\mt{F}$};
\filldraw [red] (0.1,1.5) circle (0pt) node[black] {$\tilde{\mt A}\mt{F}$};
\filldraw [red] (0,0) circle (0pt) node[black] {$\tilde{\mt A}\mt{E}$};
 \filldraw [red] (-0.1,-1.5) circle (0pt) node[black] {$\tilde{\mt A}\mt{F}$};
\filldraw [red] ({0.8+0.05},1.5) circle (0pt) node[black] {$\tilde{\mt B}\mt{F}$};
\filldraw [red] (0.8,0) circle (0pt) node[black] {$\tilde{\mt B}\mt{E}$};
\filldraw [red] (0.8,-1.5) circle (0pt) node[black] {$\tilde{\mt B}\mt{F}$};

\draw[-,thick,black] ({-2+6},0) -- ({-1.2+6},0);
\draw[-,thick,black] (6,-3.5) -- (6,-3);
\draw[->,thick,black] ({1.2+6},0) -- ({2+6},0);
\draw[->,thick,black] (6,3) -- (6,3.8);

\draw[<->,thin,black] (-0.4,{3-0.1}) -- (0.4,{3-0.1});
\filldraw [black] ({0},{3-0.1}) circle (0pt) node[anchor=north,black] {\tiny $2h\lambda$};
\draw[<->,thin,black] (-0.4,{-3-0.1}) -- (-1.2,{-3-0.1});
\filldraw [black] ({-0.8},{-3-0.1}) circle (0pt) node[anchor=north,black] {\tiny $h(1-\lambda)$};

\draw[<->,thin,black] ({1.2+0.2},{3}) -- ({1.2+0.2},{0.7});
\filldraw [black] ({1.2+0.2},{1.85}) circle (0pt) node[anchor=west,black] {${1-\mu}$};
\draw[<->,thin,black] ({1.2+0.2},{-3}) -- ({1.2+0.2},{-0.7});
\filldraw [black] ({1.2+0.2},{-1.85}) circle (0pt) node[anchor=west,black] {${1-\mu}$};

\filldraw [black] ({2+6},0) circle (0pt) node[anchor=north,black] {$x_1$};
\filldraw [black] (6,3.8) circle (0pt) node[anchor=west,black] {$x_2$};

\filldraw [black] ({1.2+6},0) circle (0pt) node[anchor=north west,black] {$h$};
\filldraw [black] ({-1.2+6},0) circle (0pt) node[anchor=north east,black] {$-h$};
\filldraw [black] (6,3) circle (0pt) node[anchor=south west,black] {$1$};
\filldraw [black] (6,-3) circle (0pt) node[anchor=north west,black] {$-1$};

\draw[-,thin,black] ({-1.2+6},0) -- ({0+6},3);
\draw[-,thin,black] ({1.2+6},0) -- ({0+6},3);
\draw[-,thin,black] ({-1.2+6},0) -- ({0+6},-3);
\draw[-,thin,black] ({1.2+6},0) -- ({0+6},-3);

\draw[-,thin,black] ({1.2+6},0) -- ({0.4+6+0.2},0.7);
\draw[-,thin,black] ({1.2+6},0) -- ({0.4+6-0.2},-0.7);
\draw[-,thin,black] ({-1.2+6},0) -- ({-0.4+6+0.2},0.7);
\draw[-,thin,black] ({-1.2+6},0) -- ({-0.4+6-0.2},-0.7);

\draw[-,thin,black] ({6},3) -- ({0.4+6+0.2},0.7);
\draw[-,thin,black] ({6},-3) -- ({0.4+6-0.2},-0.7);
\draw[-,thin,black] ({+6},3) -- ({-0.4+6+0.2},0.7);
\draw[-,thin,black] ({+6},-3) -- ({-0.4+6-0.2},-0.7);

\draw[-,thin,black] ({-0.4+6+0.2},0.7) -- ({0.4+6+0.2},0.7);
\draw[-,thin,black] ({-0.4+6-0.2},-0.7) -- ({0.4+6-0.2},-0.7);
\draw[-,thin,black] ({-0.4+6-0.2},-0.7) -- ({-0.4+6+0.2},0.7);
\draw[-,thin,black] ({0.4+6+0.2},0.7) -- ({0.4+6-0.2},-0.7);
\draw[black,fill=red]  ({0.4+6+0.2},0.7) -- ({1.2+6},0) -- (6,3) -- cycle;
\draw[black,fill=red]  ({6-(0.4+0.2)},-0.7) -- ({-1.2+6},0) -- (6,-3) -- cycle;
\draw[black,fill=yellow]  ({-0.4+6+0.2},0.7) -- ({-1.2+6},0) -- (6,3) -- cycle;
\draw[black,fill=yellow]  ({+0.4+6-0.2},-0.7) -- ({1.2+6},0) -- (6,-3) -- cycle;
\draw[black,fill=orange]  ({+0.4+6-0.2},-0.7) -- ({-0.4-0.2+6},-0.7) -- (6,-3) -- cycle;
\draw[black,fill=orange]  ({-0.4+6+0.2},0.7) -- ({+0.4+0.2+6},0.7) -- (6,3) -- cycle;
\draw[black,fill=green]  ({-0.4+6+0.2},0.7) -- ({-1.2+6},0) -- ({-0.4+6-0.2},-0.7) -- cycle;
\draw[black,fill=green]  ({0.4+6+0.2},0.7) -- ({1.2+6},0) -- ({0.4+6-0.2},-0.7) -- cycle;
\draw[black,fill=blue]  ({0.4+6+0.2},0.7) -- ({0.4+6-0.2},-0.7) -- ({-0.4+6-0.2},-0.7) -- ({-0.4+6+0.2},0.7) -- cycle;

\filldraw [black] ({0.4+0.2},0.7) circle (1pt) node[anchor=south east,black] {$\vc x_b$};
\filldraw [black] ({-0.4+0.2},0.7) circle (1pt) node[anchor=south east,black] {$\vc x_a$};
\filldraw [black] ({0.4-0.2},-0.7) circle (1pt) node[anchor=north west,black] {$\vc x_c$};
\filldraw [black] ({-0.4-0.2},-0.7) circle (1pt) node[anchor=north west,black] {$\vc x_d$};

\filldraw [black] ({0.4+0.2+6},0.7) circle (1pt) node[anchor=south east,black] {$\vc x_b$};
\filldraw [black] ({-0.4+0.2+6},0.7) circle (1pt) node[anchor=south east,black] {$\vc x_a$};
\filldraw [black] ({0.4-0.2+6},-0.7) circle (1pt) node[anchor=north west,black] {$\vc x_c$};
\filldraw [black] ({-0.4-0.2+6},-0.7) circle (1pt) node[anchor=north west,black] {$\vc x_d$};
\end{tikzpicture}
\caption{
\label{ContiCostruction}
Construction of the deformations $\vc v$ and $\tilde{\vc u}$ as in Lemma \ref{lem:Conti}. Here, we used the notation $x_i = \vc x\cdot {\vc e}_i$ for $i=1,2.$ In the left panel, we have illustrated the gradients of $\vc v$, which is a composition of horizontal and vertical laminates. On the right different colors correspond to different level-sets for $\nabla\tilde{\vc u}$: In blue, the set where $\nabla\tilde{\vc u}=\tilde{\mt A}\mt{E}$, in green where $\nabla\tilde{\vc u}=\tilde{\mt B}\mt{E}$, in orange where $\nabla\tilde{\vc u}=\tilde{\mt A}\mt{F}$, and in red and yellow the interpolation layers where $\nabla\tilde{\vc u}\in\{\mt G,\mt H\}$. Outside $\mt R^T\Omega_{h,\mt R}:=\conv(\{\pm h {\vc e}_1,\pm {\vc e}_2\})$ we have $\nabla\tilde{\vc u} = \mt{Id}$. Here $\vc x_a,\vc x_b,\vc x_c,\vc x_d$ are respectively given by $(-\lambda h + q \mu (1-\mu),\mu),$ $(\lambda h +  q \mu (1-\mu),\mu),$ $(\lambda h -  q \mu (1-\mu),-\mu),$ $(-\lambda h - q \mu(1-\mu),-\mu)$.
}
\end{figure}

In order to obtain (iv) outside of the interpolation region, we estimate
\begin{align*}
|\tilde{\mt A}\mt E-\tilde{\mt A}| &\leq |\tilde{\mt A}||\mt E-\mt{Id}| \leq \sqrt{3} q (1-\mu),\qquad
|\tilde{\mt A}\mt F- \tilde{\mt A}| \leq |\tilde{\mt A}||\mt F-\mt {Id}| \leq \sqrt{3} q \mu,\\
|\tilde{\mt B}\mt E-\tilde{\mt B}| &\leq |\tilde{\mt B}||\mt E-\mt{Id}| \leq \sqrt{3} q (1-\mu),\qquad
|\tilde{\mt B}\mt F-\tilde{\mt B}|\leq |\tilde{\mt B}||\mt F-\mt{Id}| \leq \sqrt{3} q \mu.
\end{align*}
Choosing $\mu = (1-\lambda) h$ and recalling that $\lambda \leq 1-\lambda$ implies the claimed estimate for the non-interpolation regions (orange, blue and green regions in Figure \ref{ContiCostruction}). For the interpolation regions we start by noticing that  
\begin{align*}
\mt G = \tilde{\mt B} - \frac{\lambda h}{1-(1+\lambda)h} \begin{pmatrix} -1 & -h \\ \lambda +1 & 1  \end{pmatrix},\qquad
\mt H = \tilde{\mt B} - \frac{\lambda h}{1-(1-\lambda)h} \begin{pmatrix} -1 & h \\ \lambda -1 & 1  \end{pmatrix}
.
\end{align*}
As a consequence, 
\begin{align*}
|\mt G-\tilde{\mt B}| \leq C \lambda h,\qquad|\mt H-\tilde{\mt B}| \leq C \lambda h.
\end{align*}
With this in hand, the estimate from (v) follows from the fundamental theorem. Assume without loss of generality that $\vc x\cdot {\vc e}_1\geq0$ (the other case can be treated similarly), then
\begin{align*}
\frac1\rho|\vc u(\mt R\vc x) - \mt C\mt R\vc x| 
&= |\tilde{\vc u}(\vc x)-  \vc x| \leq \int\limits_{0}^h|\nabla \tilde{\vc u}(\vc x + t{\vc e_1}) - \mt C| dt\\ 
&\leq C_1\left(|\Omega_A^{\vc x}| (1-\lambda) + |\Omega_B^{\vc x}|\lambda \right),
\end{align*}
where $\Omega_{\mt A}^{\vc x}, \Omega_{\mt B}^{\vc x}$ denote the subintervals of the segment connecting $\vc x$ and $\vc x+h{\vc e}_1$ in which $\nabla \vc u$ is closer to $\mt A$ or $\mt B$ respectively. Noting that $|\Omega_{\mt A}^{\vc x}| \leq C_1 \lambda h $ and $|\Omega_{\mt B}^{\vc x}|\leq C_1 (1-\lambda)h$ then concludes the argument.

Last but not least, the estimate on the volume fraction of $\Omega_{\mt A}$ and $\Omega_{\mt B}$ follows from the observations that 
\begin{itemize}
\item  the whole diamond has volume $2h$,
\item the set $\Omega_{\mt A}$ is given by the orange and blue domains in Figure \ref{ContiCostruction} which has volume $|\Omega_{\mt A}| = 2 \frac{1}{2}(2h \lambda)(1-\mu) + 2h \lambda 2\mu 
= 2 h\lambda (1+\mu)
= 2 h \lambda (1+(1-\lambda)h).
$
\end{itemize} 
As a consequence, we obtain the claims in (vi).
\end{proof}

\subsection{Covering in diamond shaped-domains}
With Lemma \ref{lem:Conti} in hand, we can formulate a replacement construction satisfying strong perimeter bounds in the special diamond-shaped domains $\Omega_{h, \mt R}$ with $h\in (0,8^{-1})$, $\mt R \in SO(2)$ which we had introduced in the previous section. As in \cite{RZZ18}, we discuss the cases with boundary conditions $\mt M \in \mathcal{V}_j$, $j\in\{0,1\}$, and $\mt M \in \mathcal{V}_j$, $j\geq 2$, separately. We begin by introducing a replacement construction for deformation gradients $\mt M \in \mathcal{V}_j$ with $j\in \{0,1\}$.

\begin{lem}
\label{lem:diamond_replace}
Let $\mt M \in \mathcal{V}_{j}$ with $j\in \{0,1\}$. Then there exist
\begin{itemize}
\item a diamond-shaped domain $\Omega_{h,\mt R}$ with $h = h(\mt M)$,
\item a piecewise affine deformation $\vc u: \Omega_{h,\mt R} \rightarrow \R^2$,
\end{itemize}
such that
\begin{itemize}
\item[(i)] $\nabla \vc u(\vc x) \in \mathcal{V}_{k}$ a.e. in $\Omega_{h,\mt R}$ for some $k>j$;
\item[(ii)] $\vc u(\vc x) = \mt M \vc x $ on $\partial \Omega_{h,\mt R}$;
\item[(iii)] the level sets of $\nabla \vc u$ consist of four pairs of disjoint open triangles and one parallelogram (which can again be split into two open triangles up to a set of measure zero);
\item[(iv)] if $T_1,\dots, T_{10}$ denote the open triangles from (iii) forming the level sets of $\nabla \vc u$ (after splitting up the parallelogram), then
\begin{align*}
\sum\limits_{j=1}^{10} \Per(T_j) \leq 10 \Per(\Omega_{h,\mt R}).
\end{align*}
\end{itemize}
\end{lem}

\begin{proof}

We split the proof into two cases:

\textit{Case $j=1$. } Let $m \geq 1$ be the integer such that $\left|\mathsf{C}(\mathsf{F}) \mathbf{e}_1 \cdot \mathbf{e}_1 - 1 \right|\in 2^{-(m+1)}\zeta_0 \cdot (1,2)$.
In this case we move from $\mathcal{V}_1$ into $\mathcal{V}_{2m+1}$ by improving on the $\mt C(\mt F) \vc e_2 \cdot \vc e_2$ component. More precisely, setting $\mt F = \mt R \mt F_2(\mu, \lambda)$ with $\mt R \in SO(2)$, $\mu, \lambda \in (0,1)$, using Lemma \ref{lemma3}, we split
\begin{align*}
\mt F = \mt R \mt F_2(\mu, \lambda)
= \rho \mt R\mt F_2(\mu^{\ast},\lambda) + (1-\rho) \mt R\mt F_2(1-\mu^{\ast}, \lambda)
\end{align*}
with $\rho, \lambda, \mu \in (0,1)$ and where
\begin{align*}
\mt C(\mt F_2(\mu^{\ast},\lambda)) \vc e_2 \cdot \vc e_2 = (1+\delta^2) - \frac{3}{4} \zeta_0 2^{-m}.
\end{align*}
We recall that along such a splitting using the ``$\mt F_2$ coordinates'' the $\mt C(\mt F_2) \vc e_1 \cdot \vc e_1$ component remains unchanged, that is 
$$
\mt C(\mt F_2(\mu^{\ast},\lambda)) \vc e_1 \cdot \vc e_1=\mt C(\mt F) \vc e_1 \cdot \vc e_1 = \mt C(\mt F_2(1-\mu^{\ast},\lambda))\vc e_1 \cdot \vc e_1.
$$ 
In order to obtain a suitable error estimate in the construction of Lemma \ref{lem:Conti}, we now notice that, as $K^{qc}$ is a bounded compact set, and as $\mt C(\mt F)$ is Lipschitz in $\mt F$ in any bounded set, there exists $C_L\geq 1$ depending on $\delta$ only such that
\beq
\label{eq:LipC}
|\mt C(\mt A)-\mt C(\mt B)|\leq C_L |\mt A-\mt B|, \eeq
for all $\mt A,\mt B\in K^{qc}.$ 
We choose $h>0$ (depending on $\mt M$) so small that 
\begin{itemize}
\item for the error estimate in Lemma \ref{lem:Conti} (iv) and the constant $C_L>0$ from \eqref{eq:LipC} we have 
\begin{align*}
e:= C_1 |\mt F_2(1-\mu^{\ast},\lambda)- \mt F_2(\mu^{\ast},\lambda)| \min\{\rho, 1-\rho\}h_1 < \zeta_0  \frac{\min\{1,\beta\}}{C_L 2^{m+3}},
\end{align*}
where 
\begin{align*}
\beta:= \min\left\{\left|1+\delta^2 - 2^{-m}\zeta_0 - \mt C(\mt F) \vc e_1 \cdot \vc e_1 \right|,\left|1+\delta^2 - 2^{-m+1}\zeta_0 -\mt C(\mt F) \vc e_1 \cdot \vc e_1 \right|\right\},
\end{align*}
and where $C_1$ is as in Lemma \ref{lem:Conti},
\item for the deformation gradient $\nabla \vc u$ from Lemma \ref{lem:Conti} we have 
\begin{align*}
|\mt C(\nabla \vc u) \vc e_1 \cdot \vc e_1 - 1| \in 2^{-(m+1)} \zeta_0 \cdot (1,2).
\end{align*}
\end{itemize}
Using this error bound together with the estimate (iv) in Lemma \ref{lem:Conti} and \eqref{eq:LipC} gives a diamond shape domain $\Omega_{h,\mt R}$ and a map $\vc u:\Omega_{h,\mt R}\to \R^2$ such that
$$
|\mt C(\nabla \vc u) - \mt C(\mt  F_2(\mu^*,\lambda))| \leq \bar C |\nabla \vc u - \mt  F_2(\mu^*,\lambda)| \leq \zeta_0  \frac{\min\{1,\beta\}}{2^{m+2}},
$$
which implies that $\nabla \vc u(\vc x) \in \mathcal{V}_{2m+1}$ a.e. in $\Omega_{h,\mt R}$. The remaining properties (ii)-(iv) are direct consequences of the corresponding properties of the construction from Lemma \ref{lem:Conti}.

\textit{Case $j=0$. } In this case there exist $m_1,m_2\in \Z$ such that
\begin{align*}
&|\mt C(\mt F) \vc e_1 \cdot \vc e_1 - 1| \in \zeta_0 [2^{-(m_1+1)}, 2^{-m_1}],\\
&|\mt C(\mt F) \vc e_2 \cdot \vc e_2 - (1+\delta^2)| \in \zeta_0 [2^{-(m_2+1)}, 2^{-m_2}].
\end{align*}
We set $\bar{m}= \max\{m_1,m_2\}$ and let $m \geq 1$ be defined as the smallest positive integer strictly larger than $\bar{m}$. We then carry out a splitting improving $\mt C(\mt F) \vc e_1 \cdot \vc e_1$. More precisely, we set $\epsilon = \frac{3}{4}\zeta_0 2^{-m}$ and invoke Lemma \ref{lemma2} to obtain
\begin{align*}
\mt F = \mt R \mt F_1(\mu, \lambda)
= \rho \mt F_1(\mu^{\ast}, \lambda) + (1-\rho) \mt F_1(1-\mu^{\ast},\lambda)
\end{align*}
such that $\mt C(\mt F_1) \vc e_1 \cdot \vc e_1 = \mt C(\mt F_1) \vc e_1 \cdot \vc e_1 = 1-\epsilon$. Invoking the replacement construction from Lemma \ref{lem:Conti} and choosing the error $e$ (depending on $\mt F$) sufficiently small (cf. case $j=1$) then yields a new deformation such that $\nabla \vc u \in \mathcal{V}_1$.
\end{proof}

We emphasize that a key requirement in the replacement construction of Lemma \ref{lem:diamond_replace} is to ensure that $\nabla \vc u \in \mathcal{V}_{l}$ with $l>j$ a.e. in $\Omega_{h, \mt R}$. Relying on the replacement construction of Lemma \ref{lem:Conti}, this necessitates that $h$ is allowed to depend on $\mt M$. In particular, this yields highly non-uniform estimates on the aspect ratios of the diamond-shaped domains $\Omega_{h, \mt R}$. This will also be reflected in our argument on general covering results (see Section \ref{sec:Covering}).
\medskip

If $\mt M \in \mathcal{V}_k$ with $k\geq 2$, we infer a more detailed result. In particular, on the one hand, the bounds in (iv) will be crucial in obtaining strong geometric $L^1$ convergence. On the other hand, the possibility of choosing $h_0$ \emph{uniformly}, independently of $\mt M$, will allow us to obtain exponential (instead of super-exponential) $BV$ estimates.

\begin{lem}
\label{lem:diamond}
There is a universal constant $h_0\in (0,8^{-1})$ such that for any $\mt M \in \mathcal{V}_k$ with $k\geq 2$ there exist
\begin{itemize}
\item a diamond-shaped domain $\Omega_{h_0, \mt R}$,
\item a piecewise affine deformation $\vc u: \Omega_{h_0, \mt R} \rightarrow \R^2$,
\item a domain $\Omega_{h_0, \mt R}^{\star}\subset \Omega_{h_0, \mt R} $ consisting of a union of level sets of $\nabla \vc u$,
\item constants $v_0 \in (0,1)$ and $c\geq 1$ independent of $\mt M$ and $h_0$ with $c v_0  \in (0,1)$, 
\end{itemize}
such that
\begin{itemize}
\item[(i)] $\nabla \vc u(\vc x) \in \mathcal{V}_{k+1}$ for almost every $\vc x \in \Omega_{h_0, \mt R}$,
\item[(ii)] $\vc u( \vc x) = \mt M \vc x $ on $\partial \Omega_{h_0, \mt R}$,
\item[(iii)] (up sets of measure zero) the level sets of $\nabla \vc u$ consist of 10 open triangles,
\item[(iv)] $|\nabla \vc u - \mt M| \leq c 2^{-\left\lfloor \frac{k}{2}\right\rfloor}$ on $\Omega_{h_0, \mt R}^{\star}$ and 
$|\Omega_{h_0, \mt R}^{\star}| \geq (1-c v_0^{\left\lfloor \frac{k}{2}\right\rfloor}) |\Omega_{h_0, \mt R}|$,
\item[(v)] if $T_1,\dots, T_{10}$ denote the open triangles forming the level sets of $\nabla \vc u$, then
\begin{align*}
\sum\limits_{j=1}^{10} \Per(T_j) \leq 20 \Per(\Omega_{h_0, \mt R}).
\end{align*}
\end{itemize}
\end{lem}

\begin{proof}
 We first assume that $k\geq 2$ is even, and we denote $\bar{k}=\frac{k}{2}$.
Let $\eps:= \frac{3}{4}2^{-\bar{k}}\zeta_0$.
Given $\mt M \in \mathcal{V}_k$, by Lemma \ref{lemma3}
there exists $\mt Q \in SO(2)$ with $\mt M = \mt Q \mt F_2(\mu, \lambda)$ and
\begin{align*}
\mt M = \rho \mt Q \mt F_2(\mu^{\ast},\lambda) + (1-\rho) \mt Q \mt F_2(1-\mu^{\ast},\lambda)
\end{align*}
such that
\begin{align*}
\mt C_2(\mu^{\ast},\lambda) \vc e_2 \cdot \vc e_2 
= \mt C_2(1-\mu^{\ast},\lambda) \vc e_2 \cdot \vc e_2
= 1+\delta^2 - \eps. 
\end{align*}
In particular, by our choice of $\eps$, we obtain that
 for $\mt N \in \{\mt C( \mt F_2(\mu^{\ast}, \lambda)), \mt C( \mt F_2(1-\mu^{\ast},\lambda))\}$ we have
\begin{align*}
|\mt N_{11}-1| \in \zeta_0 \cdot (2^{-\bar{k}-1},2^{-\bar{k}}), \ \mt N_{22} = 1+\delta^2 - 3 \zeta_0 2^{-\bar{k}-2}.
\end{align*}
Since $\mt M \in \mathcal{V}_k$ and by the explicit expressions for the Cauchy-Green tensor \eqref{C2Mat}, we may assume that $\lambda(1-\lambda) \in (0,2^{-6})$ and 
\begin{align*}
\epsilon_1 := (1+\delta^2) - \mt C_2(\mu, \lambda) \vc e_2 \cdot \vc e_2 \in (0, 2^{-\bar{k} +1}\zeta_0 ) .
\end{align*}
Without loss of generality, we may further suppose that $\mu < \frac{1}{2}$, whence, by \eqref{eq:control}, we also infer that for a universal constant $\bar{C}_1 
\geq1$
\begin{align*}
\left|\rho- \frac{1}{2}(\chi + 1) \right|\leq \bar{C}_1 \epsilon_1
= 2 \bar{C}_1 \zeta_0 2^{-\bar{k}},
\end{align*}
where $\chi$ denotes the function from Lemma \ref{lemma3}.
Choosing 
\begin{align*}
h_0:= \left(64\delta(1+\delta)^2 C_1 \bar{C}_1 C_L\right)^{-1}> 0
\end{align*}
with $C_1 \geq 1$ being the constant from Lemma \ref{lem:Conti} (iv), $C_L$ being as in \eqref{eq:LipC}, and applying the construction from Lemma \ref{lem:Conti} then yields that the error in the replacement construction from Lemma \ref{lem:Conti} (iv) satisfies
\begin{align*}
e \leq C_1 |\mt F_2(\mu^{\ast},\lambda) - \mt F_2(1-\mu^{\ast},\lambda)| {\left|\rho- \frac{1}{2}(\chi + 1) \right|} h_0
\leq \frac{2^{-\bar{k}}\zeta_0}{16{C_L}}.
\end{align*}
Here we used also the fact that $|\mt F_2(\mu^{\ast},\lambda) - \mt F_2(1-\mu^{\ast},\lambda)| \leq 2\delta(1+\delta)^2$.
Thus, by property (iv) in Lemma \ref{lem:Conti} together with \eqref{eq:LipC}, we infer that $\nabla \vc u \in \mathcal{V}_{k+1}$ a.e. in $\Omega_{h_0,\mt R}$. The claimed properties (ii), (iii) and (v) then also follow from the corresponding statements from Lemma \ref{lem:Conti}. It remains to argue that (iv) holds. To this end, without loss of generality, we assume that $\chi=1$ and note that in 
\begin{align*}
\Omega_{\mt Q \mt F_2(\mu^{\ast},\lambda)}:= \{\vc x \in \Omega_{h_0, \mt R}: \  |\nabla \vc u(\vc x)- \mt Q \mt F_2(\mu^{\ast},\lambda)|<| \nabla \vc u(\vc x)- \mt Q \mt F_2(1-\mu^{\ast},\lambda)|\}
\end{align*}
we have 
\begin{align*}
|\mt M - \nabla \vc u(\vc x)|
&\leq |\mt M - \mt Q \mt F_2(\mu^{\ast},\lambda)| + |\mt Q \mt F_2(\mu^{\ast},\lambda) - \nabla \vc u(\vc x)|
\\
&\leq |\mt M - \mt Q \mt F_2(\mu^{\ast},\lambda)| + \frac{2^{-\bar{k}}\zeta_0}{8}
\\
&\leq (1-\rho)(|\mt F_2(\mu^{\ast},\lambda)| + |\mt F_2(1-\mu^{\ast},\lambda)|)+ \frac{2^{-\bar{k}}\zeta_0}{8}\\
&\leq c 2^{-\bar{k}},
\end{align*}
which proves the first part of the claim in (iv) with $\Omega_{h_0, \mt R}^{\star}:=\Omega_{\mt Q \mt F_2(\mu^{\ast},\lambda)}$. The claim on the volume fractions then follows from the first estimate in Lemma \ref{lem:Conti} (vi) with $\lambda = \rho$. Indeed,
$$
\rho\left( 1 + (1-\rho) h\right)\geq 1 - |\rho-1|\geq 1 - c v_0^k.
$$

If $k\geq 2$ is odd, then we may argue similarly as above where the splitting using the $\mt F_2$ coordinates is replaced by the splitting using the $\mt F_1$ coordinates (using Lemma \ref{lemma2} instead of Lemma \ref{lemma3} in order to improve the $\mt C(\mt F) \vc e_1 \cdot \vc e_1$ component). Since for any $\mt F \in \mathcal{V}_{k}$ (with $k\geq 2$ odd), the $\mt C(\mt F) \vc e_2 \cdot \vc e_2$ component is contained in $\zeta_02^{-\bar{k}-3} \cdot ( 5,7) \subset \zeta_0 \cdot (2^{-\bar{k}-1},2^{-\bar{k}})$, with $\bar{k}=\frac{k-1}{2}$, by choosing $h_0$ sufficiently small (but universal, cf. with case $k$ even) the replacement gradient from Lemma \ref{lem:Conti} is contained in $\mathcal{V}_{2k+1}$ a.e..
\end{proof}

\section{Covering}
\label{sec:Covering}
Having established the covering in the diamond-shaped domains from the previous section, we recall the two-dimensional covering algorithm with good perimeter bounds which had been introduced in \cite{RZZ18} (see Section 6.1 therein). Although the arguments leading to Lemma \ref{lem:A3} are generic properties of coverings in $\R^2$ and do not depend on the phase transformation under consideration, we recall their proof for self-containedness. 

We remark that the set of diamond-shaped domains is not ``closed'' with respect to iterating the construction: the level sets within the diamond-shaped domains which are covered in the next iteration step are triangles and not again diamond-shaped domains. Hence, any ``closed'' set with respect to the iteration of our algorithm has to contain the set of resulting triangles. As a consequence we define the following families of sets (see Figures \ref{fig:cover1}, \ref{fig:cover2}):

\begin{defi}
\label{defi:cover}
We denote by $\mathcal{C}$ the set of all possible open triangles. 
Let $\Omega_{h,\mt R}\subset\R^2$ be a rhombus with two orthogonal axes of length $2, 2h$ respectively, with its long axis pointing into the direction $\mt R\vc e_2$. Then, by $\mathcal{C}_{\mt R,h}$ we denote the set of open isosceles triangles of angles $\arctan \frac{1}{h},\arctan \frac{1}{h},\pi- 2\arctan \frac{1}{h}$, and which are symmetric about $\mt R\vc e_2$.
\end{defi}

We deal first with the case in which the data $\mt M$ (to be replaced with a deformation gradient closer to the wells) are not yet in a dyadic neighbourhood of the wells, i.e., where $\mt M \in \mathcal{V}_j$ with $j\in\{0,1\}$. 

\begin{lem}[Validity of condition (A3) in \cite{RZZ18}]
\label{lem:A3}
Let $\mt M\in \mathcal{V}_j$, $j\in\{0,1\}$ and let $T \in \mathcal C$ be an open triangle. Then there exist 
\begin{itemize}
\item a piecewise affine function $\vc u:T \rightarrow \R^2$ whose gradient only attains finitely many values and whose level sets consist of the triangular domains  $T_1,\dots, T_N$, 
\item a domain $T_g \subset T$ which is itself a union of finitely many open triangles (up to a set of measure zero),
\item and a constant $v\in(0,1)$, which is independent of $\mt M$,
\end{itemize}
such that
\begin{align}
\label{eq:improve}
\vc u( \vc x) &= \mt M \vc x \mbox{ on } \partial T,\
\nabla \vc u \in \mathcal{V}_{k} \mbox{ in } T_g,\
|T_g|\geq v |T|,
\end{align}
for some $k>j$. Moreover, there exists $\mt R\in SO(2), h\in(0,\frac18)$ (depending on $\mt M$) such that the following perimeter bounds hold: 
\begin{itemize}
\item[(i)]
If $T\in \mathcal{C}\setminus \mathcal{C}_{\mt R, h}$ there exists a constant $C_0=C_0(\mt M)>1$ which is independent of $T$ such that
\begin{align*}
\sum\limits_{\widetilde{T} \subset T_g, \ \widetilde{T} \in \{T_1,\dots,T_N\}} \Per(\widetilde{T}) \leq C_0 \Per(T).
\end{align*} 
\item[(ii)] If $T \in \mathcal{C}_{\mt R,h}$, then there exists a uniform constant $C_2>1$ (independent of $\mt M$ and $T$) such that
\begin{align*}
\sum\limits_{ \widetilde{T} \in \{T_1,\dots, T_N\}} \Per(\widetilde{T})
\leq C_2 \Per(T).
\end{align*}
\item[(iii)] If $T\in \mathcal{C}\setminus \mathcal{C}_{\mt R,h}$, then there exists an up to null-sets disjoint splitting of $T\setminus T_g:= T^{[1]}\cup T^{[2]}$ with $T^{[1]} \subset \mathcal{C}_{\mt R,h}$, $T^{[2]}\subset \mathcal{C} \setminus \mathcal{C}_{\mt R,h}$ such that
\begin{align*}
&\sum\limits_{\widetilde{T}\subset T^{[1]}, \ \widetilde{T} \in \{T_1,\dots,T_N\}} \Per(\widetilde{T}) \leq C_0 \Per(T),\\
& \sum\limits_{\widetilde{T} \subset T^{[2]}, \ \widetilde{T} \in \{T_1,\dots,T_N\}} \Per(\widetilde{T}) \leq C_2 \Per(T).
\end{align*}
Here $C_0, C_2>1$ are the constants from (i), (ii) respectively. In particular, $C_0$ may depend on $\mt M$, while $C_2$ is uniform.
\end{itemize}
\end{lem}

We remark that in Lemma \ref{lem:A3} the perimeter bounds are \emph{not} always uniform, since in order to move from $\mathcal{V}_0$ and $\mathcal{V}_1$ it might be necessary to use extremely degenerate building block diamond-shaped constructions as in Lemma \ref{lem:Conti}. This is taken into account in the cases (i) and (iii) above in which we allow for the constant $C_0>1$ to depend on $\mt M$. A crucial observation here is that in our convex integration algorithm (see Algorithm \ref{alg:convex_int}) this constant $C_0>1$ will only arise at most three times in our estimates, allowing us to conclude uniform exponential behaviour (independent of $\mt M$) for the $BV$ norms (see Lemma \ref{lem:BV}). 

\begin{figure}[t]
  \includegraphics[width=0.4\linewidth,page=4]{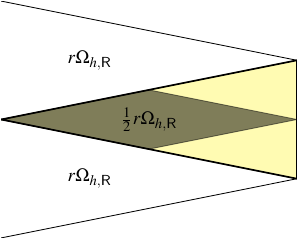}
  \includegraphics[width=0.4\linewidth,page=3]{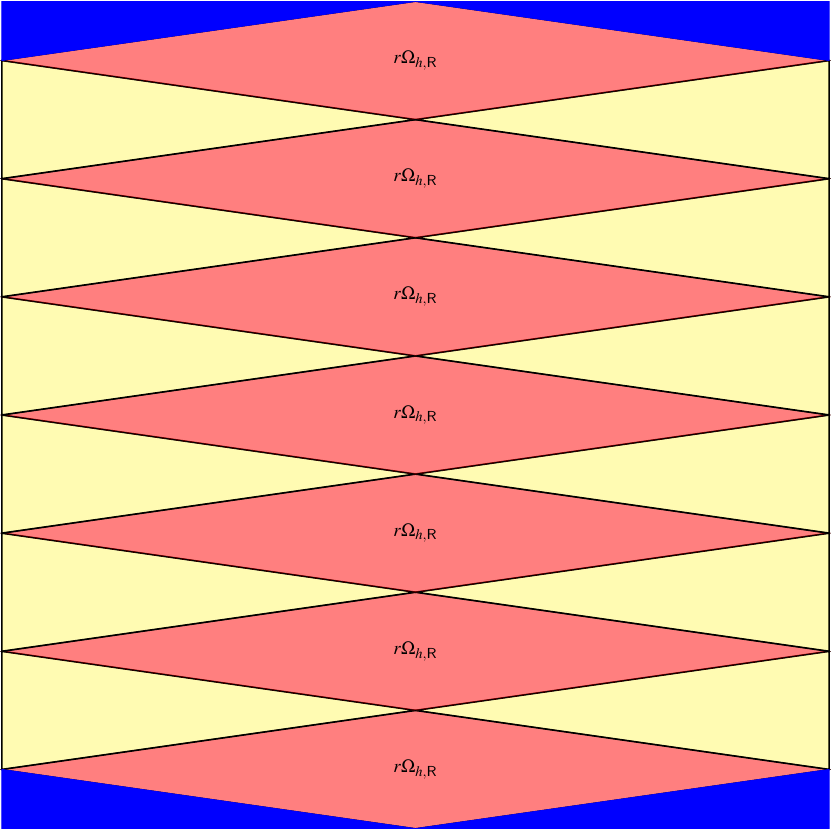}
\caption{
\label{fig:cover1}
The covering described in Steps 1 and Steps 2 in the proof of Lemma \ref{lem:A3}. The left illustration shows the covering in one of the triangles $T \in \mathcal{C}_{h,\mt R}$ (yellow) in which a diamond (green) is included in a self-similar way giving rise to two remaining triangles which are again in $\mathcal{C}_{h, \mt R}$. In the right panel, the covering of the box $r \mt R[(0,nh)\times (0,1)]$ is visualized. It is covered by $n$ diamonds (red), $2n$ triangles in $C_{h, \mt R}$ (yellow) and four additional triangles (blue). 
}
\end{figure}

\begin{proof}
Given $\mt M\in K^{qc}$ we seek to repeatedly use the replacement construction given by Lemma \ref{lem:diamond_replace} to prove the claimed result. Therefore, throughout the proof, we denote by $\mt R$ and $\Omega_{h,\mt R}$ the rotation and rhombus given by Lemma \ref{lem:diamond_replace} (with boundary data $\mt M$), and by $\mathcal{C}_{\mt R,h}$ the family of related isosceles triangles given by Definition \ref{defi:cover} (with $h$ and $\mt R \in SO(2)$ as in Lemma \ref{lem:diamond_replace}).

We split the proof into three parts: first we discuss a replacement construction for triangles $T \in \mathcal{C}_{\mt R,h}$. Next we discuss an auxiliary covering construction in a special rectangular domain. Finally in the third step, we use this to deduce the covering for general triangles.\\ 

\emph{Step 1: Replacement construction for the triangles in $\mathcal{C}_{\mt R,h}$.}
Assuming that $T$ is an isosceles triangle in $\mathcal{C}_{\mt R,h}$, one can cover $T$ by a scaled version of the diamond-shaped domain $r \Omega_{h,\mt R}$, for some $r>0$, and two remaining triangles $\widetilde{T}_1$, $\widetilde{T}_2 \in \mathcal{C}_{\mt R,h}$ (see Figure \ref{fig:cover1}, left). It is possible to choose $r$ such that half of the area of $T$ is covered by $r \Omega_{h,\mt R}$.
Within the diamond $r \Omega_{h,\mt R}$, one can apply the construction from Lemma \ref{lem:Conti} and replace the affine deformation $\mt M \vc x + \vc b$ by the piecewise affine deformation $\vc u$ from Lemma \ref{lem:diamond_replace} which fixes the boundary conditions on $\partial (r \Omega_{h,\mt R})$. In $r \Omega_{h,\mt R}$, by the construction from Lemma \ref{lem:diamond_replace}, we obtain that $\nabla \vc u \in \mathcal{V}_{l}$ with $l>j$ a.e.. Moreover $|r \Omega_{h,\mt R}| = |T_g| \geq \frac12 |T|$.
Due to the preservation of the values of the deformation on $\partial(r \Omega_{h,\mt R})$, in the remaining two triangles $\widetilde{T}_1, \widetilde{T}_2$ we do not change $\vc u$.
As the level sets of $r \Omega_{h,\mt R}$ as well as the triangles $\widetilde{T}_1, \widetilde{T}_2$ have  a perimeter which is comparable to the perimeter of $T$, we infer the bounds in (ii). \\

\emph{Step 2: Replacement construction in a rectangle $\bar R_r = r\mt R[(0,nh)\times(0,1)]$ for some $n\in\mathbb{N}$, $r>0$, and where we recall that $h=h(\mt M)$.} Stacking $n$ diamonds from Lemma \ref{lem:diamond_replace} (see Figure \ref{fig:cover1}, right) allows us to cover $\bar R_r$ with $n$ many diamonds $r\Omega_{h,\mt R}$, $2(n-1)$ many isosceles triangles in $\mathcal{C}_{\mt R,h}$, and four remaining triangles.
In the diamonds $r\Omega_{h,\mt R}$ we apply the construction from Lemma \ref{lem:diamond_replace}, in the remaining sets we keep the deformation unchanged. Defining $T_g$ to be the union of the stacked sets $r \Omega_{h,\mt R}$, the properties of the construction from Lemma \ref{lem:diamond_replace} imply that \eqref{eq:improve} holds with $v=\frac12$. The perimeter of the covering level sets $T_1,\dots,T_N$ is controlled by
\begin{align*}
\sum\limits_{\widetilde{T} \in \{T_1,\dots,T_N\}} \Per(\widetilde{T})
\leq Cn \Per(\bar{R}_r),
\end{align*}
where $C>1$ is a constant which is independent of $\mt M$, $n$, $h$, $\mt R$.
\\

\emph{Step 3: Replacement construction in a generic triangle $T$ which is not in the class $\mathcal{C}_{\mt R,h}$.}
Without loss of generality, we may assume that the triangle is right-angled (else we draw an appropriate perpendicular line and split it into two triangles where the construction below can be applied). Then, we cover (up to a null set) the triangle by a rectangle $R$ of volume $|R| \geq \frac{1}{2}|T|$ and two self-similar triangles whose perimeter is smaller than $\Per(T)$. Denoting the lengths of the orthogonal sides of $T$ by $l_1,l_2$, with $0<l_1\leq l_2$, the rectangle $R$ then has sides of length $l_1/2,l_2/2$. We cover the rectangle $R$ with $m:=\left\lfloor \frac{l_2}{l_1} \right\rfloor$ axis parallel squares $R_S$ with sides of length $l_1/2$. At least half of  $R$ can be covered by the squares $R_S$ (see Figure \ref{fig:cover2}). Finally, at least  half of each square $R_S$ can be covered by $r \tilde{\mt R} R_S$, for some $r>0$ and where $\tilde{\mt R}\in SO(2)$ brings one side of $R_S$ into $\mt R\vc e_2.$ 
Now we can apply the construction of Step 2 to cover the rectangle $\frac{r l_1}2 \tilde{\mt R}[ \left(0,h \left\lfloor \frac1h \right\rfloor\right )\times(0,1)] \subset r \tilde{\mt R} R_S$ (that is at least half of $r \tilde{\mt R} R_S$), and the rest of $r \tilde{\mt R} R_s$ can be divided into two triangles whose  perimeter is bounded by the perimeter of $R_S$. 
Therefore, in each $R_S$ we have exactly $2\left(\left\lfloor \frac1h \right\rfloor -1 \right)$ new triangles in $\mathcal{C}_{\mt R,h}\cap (T\setminus T_g)$, ten new triangles in $(\mathcal{C}\setminus\mathcal C_{\mt R,h})\cap (T\setminus T_g),$ and $10\left\lfloor \frac1h \right\rfloor$ new triangles in $T_g.$ 
Again, the perimeter of each of these triangles is bounded by the perimeter of $R_S$. Since $m\Per(R_S)\leq 4\Per(R)$ we obtain the sought result with $C_0 = 10\left\lfloor \frac1h \right\rfloor$, $C_2 = 42$ and $v=2^{-5}$. 
\end{proof}

\begin{figure}[t]
  \includegraphics[width=0.4\linewidth,page=1]{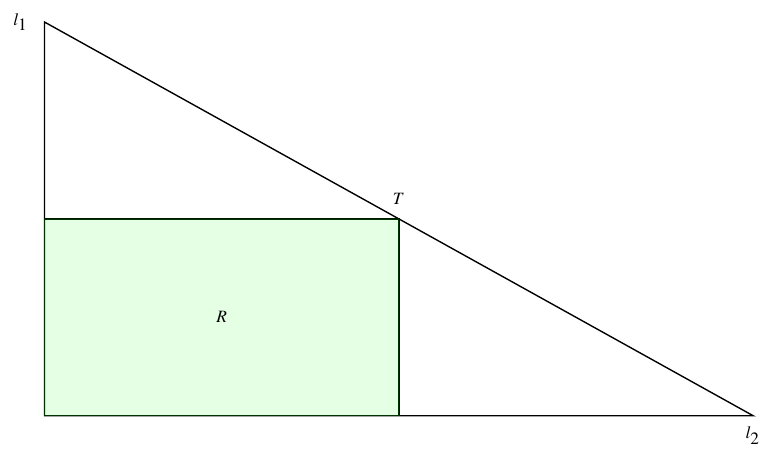}  \includegraphics[width=0.4\linewidth,page=2]{Figures.pdf}
\caption{
The covering used in Step 3 in the proof of Lemma \ref{lem:A3}. We split the right angled triangle $T$ into two self-similar triangles, and a rectangle $R$ of side lengths $l_1/2$ and $l_2/2$. This we cover by squares $R_{S}$ (of sight length $l_1/2$) which in turn are covered by rotated and rescaled versions of $R_S$, namely $r \tilde{\mt R} R_S$, which we fill (in at least half of the volume) with the construction from Step 2. 
}
\label{fig:cover2}
\end{figure}

We now consider the situation in which $\mt M \in \mathcal{V}_k$ with $k\geq 2$:

\begin{lem}[Validity of conditions (A4) and (A5) in \cite{RZZ18}]
\label{lem:A4}
Let $\mt M\in \mathcal{V}_k$, $k\geq 2$ and let $T \in\mathcal{C}$. Then there exist 
\begin{itemize}
\item a piecewise affine function $\vc u:T \rightarrow \R^2$ whose gradient only attains finitely many values and whose level sets consist of the triangular domains  $T_1,\dots, T_N$, 
\item domains $T_g, T_g^{\star} \subset T$ which are themselves a union of finitely many triangles,
\item and constants $v, v_0\in(0,1)$ and $C_3\geq 1$ with $C_3 v_0 \in (0,1)$, which are independent of $\mt M$ and $k\in \N$,
\end{itemize}
such that
\begin{align}
\label{eq:improve2}
\begin{split}
&\vc u( \vc x) = \mt M \vc x \mbox{ on } \partial T,\
\nabla \vc u \in \mathcal{V}_{k+1} \mbox{ in } T_g, \ \nabla \vc u = \mt M \mbox{ on } T \setminus T_g,\\
&|\nabla \vc u - \mt M| \leq C_3 2^{-{\left\lfloor \frac{k}{2}\right\rfloor}} \mbox{ in } T_g^{\star}, \\ 
&|T_g|\geq v |T|, \ |T_g^{\star}|\geq (1-C_3 {v_0}^{{\left\lfloor \frac{k}{2}\right\rfloor}}) |T_g|.
\end{split}
\end{align}
Moreover, the following perimeter bound holds for the new level sets: There exists a constant $C_1>0$ which is independent of $\mt M$ and $T$ such that
\begin{align*}
\sum\limits_{j=1}^{N} \Per(T_j) \leq C_1 \Per(T).
\end{align*}
\end{lem}

We remark that in contrast to Lemma \ref{lem:A3} all estimates here are uniform in $\mt M$. This allows us to obtain uniform $L^1$ decay and $BV$ growth estimates in Section \ref{sec:Algorithm}.

\begin{proof}
The proof proceeds as the one from Lemma \ref{lem:A3}, however instead of invoking Lemma \ref{lem:diamond_replace} we now rely on Lemma  \ref{lem:diamond}. In particular, the ratio $h$ of the diamond shaped domains is now fixed (and independent of $\mt M$). So all constants appearing in the covering are uniform (depending only on $h>0$ which is uniform). The set $T_g^{\star}$ is defined as the union of the sets $\Omega_{h, \mt R}^{\star}$ from the associated diamond-shaped domains. As in the proof of Lemma \ref{lem:A3} we also obtain that $|T_g| \geq \frac{1}{32}|T|$. As a consequence, the estimates in the second and third line of \eqref{eq:improve2} then follow from the corresponding statements in Lemma \ref{lem:diamond_replace}.
\end{proof}

\section{The Convex Integration Algorithm}
\label{sec:Algorithm}

In this section we combine the previous observations and deduce the desired higher Sobolev regularity of our deformations. Here we follow the general scheme from \cite{RZZ18} and combine the general covering results from the previous section with the properties of the two-well problem.

\subsection{The convex integration algorithm}

For the convenience of the reader we begin by giving a rough outline of the convex integration algorithm from \cite{RZZ18}. In the present setting of the two-well problem, it is possible to slightly simplify the algorithm.

\begin{alg}
Let $\Omega \in \R^2$ satisfy \eqref{eqDomain} and $\mt M \in \inte (K^{qc})$. Let $v\in (0,1)$ be the minimum of the corresponding constants from Lemmas \ref{lem:A3} and \ref{lem:A4}. Then, we apply the following iterative modification scheme:
\label{alg:convex_int}
\begin{itemize}
\item[(1)] \emph{Data and initialization.}
For $k\geq 0$, $k\in \N$, we consider the following data:
\begin{itemize}
\item[(i)] A piecewise affine, uniformly bounded deformation $\vc u_k: \Omega \rightarrow \R^2$.
\item[(ii)] A finite collection of sets $\hat{\Omega}_k = \{\Omega^{k,1},\dots,\Omega^{k,j_k}\} \subset \mathcal{C}$ consisting of pairwise disjoint sets whose union is (up to null-sets) $\Omega$, and such that $\nabla \vc u_k|_{\Omega^{k,j}} = const$.
\item[(iii)] An index $\ell_k: \hat{\Omega}_k \rightarrow \N$ with the property that for $\Omega^{k,j}\in \hat{\Omega}_k$ it holds that $\nabla \vc u_k|_{\Omega^{k,j}} \in \mathcal{V}_{\ell_k(\Omega^{k,j})}$.
\end{itemize}
We then initialize our construction by defining
\begin{align*}
\vc u_0( \vc x) = \mt M \vc x \mbox{ in } \Omega, 
\end{align*}
$\hat{\Omega}_0 = \{\text{triangles that cover $\Omega$ as explained in the condition \eqref{eqDomain}}\}$ and $\ell_k(T) = \{k\in \N: \ \mt M \in \mathcal{V}_k\}$ for all $T \in \hat\Omega_0$.
\end{itemize}
\item[(2)] We prescribe the algorithm inductively. Let $k \in \N$ and a triple $(\vc u_k, \hat{\Omega}_k, \ell_k)$ be given. Then, if
\begin{align*}
\begin{array}{ll}
&\ell_k(\Omega^{k,j}) \in \{0,1\}, \mbox{ we apply to $\Omega^{k,j}$ the replacement construction from Lemma \ref{lem:A3}},\\
&\ell_k(\Omega^{k,j}) >1, \mbox{ we apply to $\Omega^{k,j}$ the replacement construction from Lemma \ref{lem:A4}}.
\end{array}
\end{align*}
In both cases Lemmas \ref{lem:A3} and \ref{lem:A4} return 
\begin{itemize}
\item a replacement construction $\vc w: \Omega^{k,j}\rightarrow \R^2$ which is piecewise affine and whose gradient has finitely many, (up to null sets) pairwise disjoint level sets $\tilde{\Omega}^{k+1,1}_j,\dots, \tilde{\Omega}^{k+1,r_j}_{j}$ which are all contained in $\mathcal{C}$ and whose union (up to null sets) is $\Omega^{k,j}$,
\item a subset $\Omega^{k,j}_g$ consisting of finitely many level sets of $\nabla \vc w$ such that  
\begin{align*}
\nabla \vc w \in \mathcal{V}_{\ell_k(\Omega^{k+1,j})+1} \mbox{ in } \Omega^{k,j}_g, \ \vc w(\vc x) = \vc u_k(\vc x) \mbox{ in } (\Omega^{k,j}\setminus \Omega^{k,j}_g) \cup \partial \Omega^{k,j}.
\end{align*}
We further define $\hat{\Omega}_{k+1,j}:=\{\tilde{\Omega}^{k+1,1}_j, \dots, \tilde{\Omega}^{k+1,r_j}_{j}\}$, 
\begin{align*}
&\hat{\Omega}_{k+1} = \bigcup\limits_{j=1}^{j_k} \hat{\Omega}_{k+1,j},\
\ell_{k+1}: \hat{\Omega}_{k+1} \rightarrow \N, \ \vc u_{k+1}|_{\tilde{\Omega}^{k+1,m}_j} = \vc w|_{\tilde{\Omega}^{k+1,m}_j},\\
&\ell_{k+1}(\tilde{\Omega}^{k+1,m}_j) =
\left\{
\begin{array}{ll}
\ell_{k}(\Omega^{k,j})+1, \mbox{ if } \tilde{\Omega}^{k+1,m}_j \subset \Omega^{k,j}_g,\\
\ell_{k}(\Omega^{k,j}) \mbox{ else}.
\end{array} \right.
\end{align*}
\end{itemize}
\end{alg}

By Lemma 3.4 in \cite{RZZ18} the algorithm is well-defined and the quantity $\ell_k$ is monotone increasing. With each point $\vc x\in \Omega$ and the deformation $\vc u_k$ we then associate a characteristic function:

\begin{defi}
\label{def chis}
Let $\Omega\subset\R^2$ satisfy \eqref{eqDomain}, and let $\vc u\colon\Omega\to\R^2$ be such that $\vc u\in W^{1,\infty}(\Omega;\R^2)$. In $\Omega$ we define $\chi_{\vc u}^1\colon\R^2\to\{0,1\}$ to be the indicator function on the set
\begin{align*}
\left\{\vc x \in \Omega: \
\dist\left(\nabla\vc u(\vc x),SO(2)\mt F_0\right) \leq \dist\left(\nabla\vc u (\vc x),SO(2)\mt F_0^{-1}\right)\right\}
\end{align*}
and extend it to be zero in $\R^n \setminus \overline{\Omega}$. We further define $\chi_{\vc u}^2=1-\chi_{\vc u}^1.$ If $\vc u = \vc u_k$, where $\vc u_k\colon\Omega\to\R^2$ is the mapping obtained at the $k-$th  iteration step of the Algorithm \ref{alg:convex_int}, we write $\chi_{k}^1,\chi_{k}^2$ rather than $\chi_{\vc u_k}^1,\chi_{\vc u_k}^2$.
\end{defi}

Controlling the $L^1$ and $BV$ norms of these characteristic functions along the iteration of the convex integration algorithm of Algorithm \ref{alg:convex_int} then allows to conclude the desired higher regularity statement for their limits.

\subsection{$L^1$ estimates}

Next we recall how the properties of the replacement constructions from Lemmas \ref{lem:A3}, \ref{lem:A4} together with the explicit convex integration algorithm yield strong exponential $L^1$ decay bounds:

\begin{prop}
\label{prop:L1}
Let $\Omega\subset\R^2$ satisfy \eqref{eqDomain}. Assume further that $\vc u_k\colon\Omega\to\R^2$ is the mapping obtained at the $k-$th  iteration step of the Algorithm \ref{alg:convex_int}. Then, there exists $C>1$, $\tilde c\in(0,1)$ independent of $k$ such that
\begin{align*}
\|\chi_k^i - \chi_{k+1}^i\|_{L^1(\Omega)}\leq C \tilde{c}^k|\Omega|,\ \
\| \nabla \vc u_{k+1}- \nabla \vc u_{k} \|_{L^1(\Omega)} \leq C \tilde{c}^k|\Omega|,
\end{align*}
for $i=1,2$.
\end{prop}

\begin{proof}
We discuss the argument for the characteristic functions and the gradients simultaneously.

To this end, we first present an estimate in the $k-$th step of the algorithm for which it is enough to prove estimates for every $\tilde{\Omega}\in \hat{\Omega}_k$. Then, in a second step, we collect all the bounds to prove an $L^1-$estimate on $\Omega$. \\

\textit{Step 1: Local estimates.} Let $\tilde{\Omega}\in \hat{\Omega}_k.$ We first notice that by Lemma \ref{lem:A4} there exists $C>0$ such that
\beq
\label{step 1a}
|\nabla \vc u_{k+1}-\nabla \vc u_{k}|\leq C 2^{-\frac{\ell_k(\tilde{\Omega})}{2}}\text{ a.e. in $\tilde{\Omega}^{\star}_g$, \qquad\text{for any $k\geq 0$}. }   
\eeq
By virtue of the separateness of the wells (see Remark \ref{SeparationProperty}), we thus obtain that 
\begin{align}
\label{eq:step1a_1}
\chi^j_k =\chi^j_{k+1}\text{ a.e. in $\tilde{\Omega}^{\star}_g$, } \mbox{ whenever } \ell_k(\tilde{\Omega}) \geq k_0, \ j \in \{0,1\},
\end{align}
for some $k_0\in\mathbb{N},$ $k_0\geq 2$.
Therefore, Lemma \ref{lem:A4} together with \eqref{eq:step1a_1} yield
\begin{align*}
\|\chi_k^j-\chi_{k+1}^j\|_{L^1(\tilde{\Omega
})} 
= \|\chi_k^j-\chi_{k+1}^j\|_{L^1(\tilde{\Omega}_g)}  
= \|\chi_k^j-\chi_{k+1}^j\|_{L^1(\tilde{\Omega}_g\setminus\tilde{\Omega}_g^{\star})}
\leq |\tilde{\Omega}_g\setminus\tilde{\Omega}_g^{\star}|.
\end{align*}
Using once more Lemma \ref{lem:A4} we deduce that 
$$
|\tilde{\Omega}_g\setminus\tilde{\Omega}_g^{\star}|\leq C_3 2^{{-\frac{\ell_k(\tilde{\Omega})}{2}}}|\tilde{\Omega}|,
$$ 
that is
\beq
\label{eq2:step1}
\|\chi_k^j-\chi_{k+1}^j\|_{L^1(\tilde{\Omega
})}\leq C 2^{{-\frac{\ell_k(\tilde{\Omega})}{2}}}|\tilde{\Omega}|,
\eeq
provided $\ell_k(\tilde{\Omega})\geq k_0.$ We notice that, by taking $C = \max\{C_3,|\Omega|2^{{-\frac{k_0}{2}}}\}$, \eqref{eq2:step1} holds also when $\ell_k(\tilde{\Omega})< k_0.$\\

Similarly, \eqref{step 1a} together with the bounds from Lemma \ref{lem:A4} and the boundedness of $K^{qc}$ imply that 
\begin{align*}
\|\nabla \vc u_{k+1} - \nabla \vc u_k\|_{L^1(\tilde{\Omega})}
&\leq 
\|\nabla \vc u_{k+1} - \nabla \vc u_k\|_{L^1(\tilde{\Omega}_g)}\\
&\leq 
\|\nabla \vc u_{k+1} - \nabla \vc u_k\|_{L^1(\tilde{\Omega}_g \setminus \tilde{\Omega}_g^{\star})} + \|\nabla \vc u_{k+1} - \nabla \vc u_k \|_{L^1(\tilde{\Omega}_g^{\star})}\\
&\leq C  |\tilde{\Omega}_g \setminus \tilde{\Omega}_g^{\star}| + 2^{{-\frac{\ell_k(\tilde{\Omega})}{2}}} |\tilde{\Omega}_g^{\star}|\\
&\leq 
C 2^{{-\frac{\ell_k(\tilde{\Omega})}{2}}}|\tilde{\Omega}|,
\end{align*}
if $\ell_k(\tilde{\Omega}) \geq k_0$. Arguing similarly as above, we can then also extend this estimate to all values of $k\in \N$.\\

\textit{Step 2: Global estimate.} By applying repeatedly the estimate obtained in \eqref{eq2:step1} to every $\tilde{\Omega}\in\hat\Omega_k,$ we obtain 
\beq
\label{step2:eq1}
\|\chi_k^j-\chi_{k+1}^j\|_{L^1({\Omega
})}\leq C\sum_{\tilde\Omega\in\hat\Omega_k} 2^{{-\frac{\ell_k(\tilde{\Omega})}{2}}}|\tilde\Omega|.
\eeq
By arguing as in \cite[Lemma 4.5]{RZZ18} we deduce that 
$\mathcal{E}_k:=\sum_{\tilde\Omega\in\hat\Omega_k} 2^{{-\frac{\ell_k(\tilde{\Omega})}{2}}}|\tilde\Omega|$ satisfies
\beq
\label{step2:eq2}
\mathcal{E}_0 = |\Omega|,\qquad \mathcal{E}_{k+1}\leq \tilde{c}\mathcal{E}_k,
\eeq
for some $\tilde c\in(0,1)$ independent of $k$. Combining \eqref{step2:eq1}--\eqref{step2:eq2} we thus proved the claimed result for the characteristic functions.

For the difference of the gradients, we argue analogously and infer that
\begin{align}
\label{eq:grad_est}
\|\nabla \vc u_{k+1}- \nabla \vc u_k\|_{L^1(\Omega)}
\leq C \sum\limits_{\tilde{\Omega} \in \hat{\Omega}_k} 2^{{-\frac{\ell_k(\tilde{\Omega})}{2}}}|\tilde{\Omega}|.
\end{align}
Hence, combining \eqref{eq:grad_est} with \eqref{step2:eq2}, then implies the desired result.
\end{proof}

\subsection{$BV$ estimates}

In order to deduce the desired $W^{s,p}$ regularity of our deformation, we complement the $L^1$ estimates with $BV$ bounds. As the limiting deformations are not in BV in general, we expect $BV$ bounds which grow in $k$. Following \cite{RZZ18} we prove exponentially growing bounds:

\begin{lem}
\label{lem:BV}
Let $\Omega$ satisfy \eqref{eqDomain}. Let $\chi_k^1,\chi_k^2$ be as in Definition \ref{def chis}. Then, for $i\in \{1,2\}$
$$
| \chi_{k+1}^i-\chi_{k}^i|_{BV(\Omega)}  + | \nabla\vc u_{k+1}-\nabla\vc u_{k}|_{BV(\Omega)} \leq C \left(3\,\max\{C_1,C_2\} \right)^k|\Omega|,
$$
for some $C>1$, and where $C_1,C_2>1$ are as in Lemmas \ref{lem:A3} and \ref{lem:A4}.
\end{lem}

\begin{rmk}
In general, the constant $C>1$ depends on $\mt M$. It deteriorates as $\mt M$ approaches the boundary $\partial K^{qc}$.
\end{rmk}

In order to work with a concise notation, we recall the concept of a descendant of a domain:

\begin{defi}[Def. 3.3 in \cite{RZZ18}]
\label{def descendants}
Let $\tilde{\Omega}\in\hat\Omega_k$ for some $k\geq 0.$ Then we say that $\check{\Omega}\in\hat\Omega_l$ for some $l\geq k$ is a descendant of $\tilde{\Omega}$ if $\check{\Omega}\subset\tilde\Omega.$ We denote the set of all descendants of $\tilde{\Omega}$ by $\mathcal{D}(\tilde{\Omega})$.
\end{defi} 

We are now ready to prove Lemma \ref{lem:BV}:
\begin{proof}
We notice that, since 
\beq
\label{eq1:lem64}
\begin{split}
| \nabla\vc u_{k+1}-\nabla\vc u_{k}|_{BV(\Omega)} &\leq \| \nabla\vc u_{k+1}-\nabla\vc u_{k}\|_{L^\infty(\Omega)} \sum_{\tilde{\Omega}\in\hat\Omega_k}\Per(\tilde{\Omega}),\\
| \chi_{k+1}^i-\chi_{k}^i|_{BV(\Omega)} &\leq \sum_{\tilde{\Omega}\in\hat\Omega_k}\Per(\tilde{\Omega}),
\end{split}
\eeq
it is enough to prove an upper bound for $\sum_{\tilde{\Omega}\in\hat\Omega_k}\Per(\tilde{\Omega})$. To this aim, let $\tilde{\Omega}\in\hat\Omega_k$, we start by noticing that, thanks to Lemma \ref{lem:A3} and Lemma \ref{lem:A4} we have
$$
\sum_{\Omega'\in \mathcal{D}(\tilde{\Omega})\cap \hat\Omega_{k+1}} \delta_k(\tilde\Omega) \Per(\Omega')\leq 3 \Per(\tilde{\Omega}),
$$
where $\delta_k\colon\hat\Omega_k\to\left\{\frac{1}{C_0},\frac{1}{C_1},\frac{1}{C_2}\right\}$, and $C_0,C_1,C_2$ are as in Lemma \ref{lem:A3} and Lemma \ref{lem:A4}. As a consequence, 
\begin{align*}
\sum_{\tilde\Omega^{(1)}\in\hat\Omega_1}& \sum_{\tilde\Omega^{(2)}\in\hat\Omega_2\cap \mathcal D(\tilde\Omega^{(1)})}\dots \sum_{\tilde\Omega^{(k+1)}\in\hat\Omega_{k+1}\cap \mathcal D(\tilde\Omega^{(k)})} \delta_k(\tilde\Omega^{(1)})\dots \delta_k(\tilde\Omega^{(k)})\Per(\tilde{\Omega}^{(k+1)})
\\
&\leq 3^k \Per({\Omega}).
\end{align*}
Now, if 
$$
\min\left\{\delta_k(\tilde\Omega^{(k)})\colon {\tilde\Omega^{(k)}\in\hat\Omega_{k}\cap \mathcal D(\Omega)} \right\} \geq \min\left\{\frac{1}{C_1},\frac{1}{C_2}\right\},
$$
we have
\begin{align*}
\sum_{\tilde{\Omega}\in\hat\Omega_k}\Per(\tilde{\Omega}) 
&= \sum_{\tilde\Omega^{(1)}\in\hat\Omega_1} \sum_{\tilde\Omega^{(2)}\in\hat\Omega_2\cap \mathcal D(\tilde\Omega^{(1)})}\dots \sum_{\tilde\Omega^{(k+1)}\in\hat\Omega_{k+1}\cap \mathcal D(\tilde\Omega^{(k)})} \Per(\tilde{\Omega}^{(k+1)})
\\
&\leq \left(3\,{\max\{C_1,C_2\}}\right)^k\Per(\Omega).
\end{align*}
Otherwise, since $C_0$ may depend on the boundary condition $\mt M$, some extra work is needed. An argument as in the proof of Lemma 4.7 in \cite{RZZ18} allows one to prove that, given any sequence of sets $\tilde{\Omega}^{(k)},$ with $\tilde{\Omega}^{(k)} \in\hat\Omega_{k}\cap \mathcal D(\tilde\Omega^{(k-1)})$, there exists at most three indices $k_1,k_2,k_3\in\mathbb{N}$ such that $\delta_k(\tilde{\Omega}^{(k)}) = \frac{1}{C_0},$ for all other $k\neq k_1,k_2,k_3$ we have $\delta_k(\tilde{\Omega}^{(k)})\geq \min\left\{ \frac{1}{C_1},\frac{1}{C_2}\right\}$. Indeed, we observe that $\delta_k(\Omega_{j,k})=C_0$ only if we are in case (i) or in the first case of (iii) of Lemma \ref{lem:A3}. In the former case, we have improved $\nabla \vc u$ from being in $\mathcal{V}_{s}$ to being in $\mathcal{V}_{s+1}$ with $s\in \{0,1\}$. In the latter case, we will do so in the next step. As a consequence, 
$$
\sum_{\tilde{\Omega}\in\hat\Omega_k}\Per(\tilde{\Omega}) 
\leq C_0^{3}\left(3\,{\max\{C_1,C_2\}}\right)^{k-3}\Per(\Omega)
\leq C \left(3{\max\{C_1,C_2\}}\right)^{k} \Per(\Omega).
$$
Therefore, 
by the last identity together with \eqref{eq1:lem64} we conclude the proof of Lemma \ref{lem:BV}. 
\end{proof}

\subsection{$W^{s,p}$ bounds}
An important tool to deduce a bound on the $W^{s,p}(\Omega)$ norm for $\nabla \vc u$ is the following interpolation result (which is an adaptation of the estimates in \cite{CDDD03}, see Theorem 2, Corollary 2.1, Remark 2.2 in \cite{RZZ16} for the adaptation which we are using here):

\begin{prop}[\cite{CDDD03}]
\label{prop:Inter}
Let $u\in L^\infty(\R^n)\cap BV(\R^n)$, $n\in\mathbb{N}$, $\theta_0\in(0,1),$ $q\in(1,\infty)$ and $s>0$ such that $\theta_0 = sq$. Then,
$$
\|u\|_{W^{s,q}(\R^n)}\leq \|u\|_{L^{\infty}(\R^n)}^{1-\frac1q}\left(	\|u\|_{L^{1}(\R^n)}^{1-\theta_0}\|u\|_{BV(\R^n)}^{\theta_0}\right)^{\frac1q}.
$$ 
\end{prop}

As a consequence of this interpolation result, by Proposition \ref{prop:L1} and Lemma \ref{lem:BV} we deduce the desired $W^{s,p}$ convergence of the characteristic functions of the phases and the deformation gradients:

\begin{cor}
\label{cor:exp_dec}
Let $\Omega\subset\R^2$ satisfy \eqref{eqDomain}, and $\chi_k^j$ with $j\in \{1,2\}$ be as in Definition \ref{def chis}. There exists a constant $C>1$ and a regularity threshold $\theta_0\in(0,1)$ such that, for any $s>0,$ $p\in(1,\infty)$ satisfying $sp<\theta_0$ there exists $\alpha>0$ with
\begin{align*}
 \|\chi_k^j - \chi_{k+1}^j\|_{W^{s,p}(\Omega)} \leq C 2^{-k\alpha}, \ 
 \|\nabla \vc u_k - \nabla \vc u_{k+1}\|_{W^{s,p}(\Omega)} \leq C 2^{-k \alpha}.
\end{align*}
\end{cor}

\begin{proof}
Thanks to Proposition \ref{prop:L1}, Lemma \ref{lem:BV} and Proposition \ref{prop:Inter} it is enough to choose $\theta_0\in(0,1)$ such that
$$
\tilde{c}^{1-\theta_0}(3\max\{C_0,C_1,C_2\})^{\theta_0} = 1,
$$
where $\tilde{c}$ is as in Proposition \ref{prop:L1}, and $C_0,C_1,C_2\geq 1$ are as in Lemma \ref{lem:A3} and Lemma \ref{lem:A4}. 
Indeed, if $0<sp<\theta_0$, we obtain
\begin{align*}
\|\chi_k^i - \chi_{k+1}^i\|_{W^{s,p}(\Omega)} &= \|\chi_k^i - \chi_{k+1}^i\|_{W^{s,p}(\R^n)} 
\\
&\leq \|\chi_k^i - \chi_{k+1}^i\|_{L^{\infty}(\R^n)}^{1-\frac1p}\left(	\|\chi_k^i - \chi_{k+1}^i\|_{L^{1}(\R^n)}^{1-sp}\|v\|_{BV(\R^n)}^{sp}\right)^{\frac1p} 
\\
&\leq C \left(\tilde{c}^{1-sp}\left( 3\max\{C_0,C_1,C_2\}\right)^{sp}\right)^k.
\end{align*}
The fact that $sp<\theta_0$ thus implies that $\tilde{c}^{1-sp}\left( 3\max\{C_0,C_1,C_2\}\right)^{sp}=2^{-k \alpha}$ for some $\alpha>0$, and therefore the claim follows.
\end{proof}

Now with the previous results in hand, the proof of Theorem \ref{thm:main} follows immediately.

\begin{proof}[Proof of Theorem \ref{thm:main}]
The existence of $\nabla\vc u\in W^{s,p}(\Omega;\R^{2\times2})$ and the fact that $\nabla \vc u_k \rightarrow \nabla \vc u$ in $W^{s,p}(\Omega;\R^{2\times2})$ follows from the fact that, by Corollary \ref{cor:exp_dec}, $\nabla \vc u_k$ is a Cauchy sequence (which follows from a telescope sum argument). 

We note further that $\nabla \vc u \in K$, since for every $\epsilon>0$ and for every $k_0 \in \N$ there exists $\ell_0 \in \N$ such that for $\N \ni \ell \geq \ell_0$
\begin{align}
\label{eq:dist_improve}
|\{x\in \Omega: \ \nabla \vc u_{\ell}(x) \in \mathcal{V}_k \mbox{ for some } k \geq k_0\}|>(1-\epsilon)|\Omega|.
\end{align}
To observe this, we use the volume improvement of the good sets in Lemmas \ref{lem:A3} and \ref{lem:A4}. Indeed, after at least $k_0$ steps the volume condition that in each step $|T_g|\geq v|T|$ implies that on a set of volume of size at least $v^{k_0}|\Omega|$ we have $\nabla \vc u_{k_0} \in \mathcal{V}_{k_0}$ (or an even better set $\mathcal{V}_k$). Here we used that the function $\ell_k(\cdot)$ in Algorithm \ref{alg:convex_int} is increasing. Hence, for $m\in \N$ after $m k_0$ iteration steps at least on a volume fraction of size
\begin{align*}
\sum\limits_{j=0}^{m-1} v^{k_0}(1-v^{k_0})^j = v^{k_0}\frac{1-(1-v^{k_0})^{m}}{1-(1-v^{k_0})} = 1-(1-v^{k_0})^{m}
\end{align*}
we have $\nabla \vc u_{m k_0} \in \mathcal{V}_{k} $ for some $k \geq k_0$. Choosing $m_0\in \N$ such that $(1-v^{k_0})^{m_0}\leq \epsilon$ then yields that $\ell_0 = m_0 k_0$ satisfies the claimed estimate \eqref{eq:dist_improve}. As $\nabla \vc u \in L^{\infty}(\Omega,\R^{2\times2})$ this in particular implies that $\dist(\nabla \vc u_k , K) \rightarrow 0$ in $L^p(\Omega)$ for any $p\in [1,\infty)$ as $k\rightarrow \infty$.
\end{proof}

\section{Dimension of the singular set}
\label{sec:sing_set}

After having deduced higher regularity for our convex integration solutions for the two-well problem in the previous section, we next recall that this has direct implications on their singular sets and the regularity of the phase boundaries. In order to observe this, we invoke singular set estimates for Sobolev functions (see for instance \cite{JM96, Si}).\\

Let us recall that, given a set $A\subset\R^n$, the Box-counting dimension is defined as
\begin{align}
\label{eq:dim}
\dim_B(A) := \lim_{\eps\to 0} \frac{\log N_\eps}{\log \frac 1\eps}
\end{align}
whenever the above limit exists. Here the number $N_\eps\in\mathbb{N}$ is the number of cubes of side lengths $\eps$ needed to cover $A$. 

Not requiring the limit to exist (see for instance \cite[Chapter 5]{Mat} for a discussion on this possibility), we define 
$$
m_d(A):= \limsup_{j\to\infty} N_{2^{-j}} 2^{-jd},
$$
and 
$$
\overline{\dim}_B(A) := \inf\left\{ d\colon m_d(A) = 0  \right\} = \sup \left\{ d\colon m_d(A) = \infty  \right\}.
$$
We start by providing a lower bound on the Box dimension of $\partial \Omega$ depending on the regularity of $\chi_\Omega$. Here $\chi_{\Omega}$ denotes the characteristic function of some set $\Omega \subset \R^n$.

\begin{prop}
\label{prop:lower}
Let $\Omega\subset\R^n$ and $\theta_0\in(0,1)$ be such that $\chi_\Omega\in W^{s,p}(\R^n)$ for any $s>0,p\geq 1$ satisfying $sp<\theta_0$, but $\chi_\Omega\notin W^{s,p}(\R^n)$ for any $s>0,p\geq 1$ satisfying $sp\geq\theta_0$.   Then, $\overline{\dim}_B(\partial\Omega)\geq n-\theta_0$.
\end{prop}

\begin{proof}
Suppose for a contradiction that 
$$
d:=\overline{\dim}_B(\partial\Omega)< n-\theta_0.
$$
Then, there exists $\eps>0$ such that $m_{d+\eps}(\partial\Omega)=0$ and $d+\eps<n-\theta_0$. But then we can apply \cite[Proposition 2.1 in Chapter II]{JM96} (which works under our $\limsup$ definition) to deduce that $\chi_\Omega\in W^{s,p}(\R^n)$ with $sp=\theta_0$, contradicting  our assumptions, and thus leading to the claimed result.
\end{proof}

\begin{rmk}
Improving the optimal regularity of $\chi_\Omega$ to $W^{s,p}(\R^n)$ for any $s>0,p\geq 1$ satisfying $sp<\theta_1$, with $\theta_0<\theta_1$, decreases the lower bound  for $\dim_B(\partial\Omega)$.
\end{rmk}

\begin{rmk}
In the case that the limit in \eqref{eq:dim} exists Proposition \ref{prop:lower} turns into a statement of the true box-counting dimension.
\end{rmk}
Applying Proposition \ref{prop:lower} to the sets $\{\chi^{0}=1\}$ and $\{\chi^1 = 1\} \cap \Omega$ from Theorem \ref{thm:main}, we thus obtain that
\begin{align*}
\overline{\dim}_{B}(\partial {\{\chi^{0}=1\}}) \geq n-\theta_0, \ \overline{\dim}_{B}(\partial {\{\chi^{1}=1\}}) \geq n-\theta_0, 
\end{align*}
where $\theta_0>0$ is the exponent from Theorem \ref{thm:main}.

Given $A\subset\R^n,$ and $\rho>0$ let us now consider the thickened set 
\[
\begin{split}
S_\rho(A) :=\bigl\{x\in A\colon \exists\mu>0 \text{ s.t. }\forall\eps\in(0,1], \exists A_\eps,\tilde A_\eps, \,A_\eps\in B(x,\eps)\cap A,\,\\
 \tilde A_\eps\in B(x,\eps)\cap A^c\text{ and }|A_\eps||\tilde A_\eps|\geq \mu \eps^{2n+\rho} 	\bigr\},
\end{split}
\]
where $B(x,\eps)$ denotes the ball centred at $x$ of radius $\eps$. Denoting by $\dim_P$ the packing dimension (see \cite{Mat, Falc}) and invoking \cite[Proposition 3.3]{Si} (or
\cite[Thm 2.2]{JM96}), we arrive at
$$
\dim_P(S_\rho(\partial A))\leq n- sp.
$$
Applying this to our sets $\{\chi^{0}=1\}$ and $\{\chi^1 = 1\} \cap \Omega$, we hence obtain that
\begin{align*}
\dim_P(S_\rho(\partial \{\chi^1 = 1\})),
\dim_P(S_\rho(\partial \{\chi^1 = 0\}))
\leq n-\theta_0.
\end{align*}

\begin{rmk}
\label{rmk:sing_set}
We remark that as the definition of the set $S_\rho(\partial \{\chi^0 = 1\})$ contains a density statement, the dimension of $\partial \{\chi^0 = 1\}$ might be substantially larger than indicated by the estimate obtained above.
\end{rmk}

\section{The two-well problem with one rank-one connection}
\label{sec:one_rank_one}

In this section we consider the problem
\beq
\label{defK1}
K = SO(2)\mt F_1\cup SO(2)\mt F_2
\eeq
where $\mt F_1,\mt F_2\in\R^{2\times2}$ are respectively given by
\begin{align}
\label{eq:matrices_new}
{\mt F}_1 = \left[\begin{array}{ ccc } 1 & 0 \\ 0 & 1+\delta  \end{array}\right],\qquad
{\mt F}_2 = \left[\begin{array}{ ccc } 1 & 0 \\ 0 & 1-\delta  \end{array}\right] = \mt F_1 -2\delta \vc e_2 \otimes \vc e_2 ,
\end{align}
and $\delta \in (-1,1)$.
In this case there is only one rank-one connection between the wells (see for instance \cite[Proposition 5.1]{DM1} for this). Exploiting arguments by Ball and James \cite{BJ92}, we prove that in this phase transformation there are \emph{no} convex integration solutions (at least for affine boundary conditions) and all gradient Young measures (with affine boundary conditions) are (unique) simple laminates.\\

As a first step into this direction, we observe that the lamination convex hull is two-dimensional:

\begin{lem}
\label{lem:lc_hull}
Let $\mt F_1, \mt F_2$ be as above. Then, 
\begin{align}
\label{eq:Klc_1}
K^{lc} = K^{lc,1}= \bigcup\limits_{\lambda \in [0,1]}SO(2) \left[\begin{array}{ ccc } 1 & 0 \\ 0 & 1-\delta +2\lambda \delta \end{array}\right].
\end{align}
\end{lem}

\begin{proof}
This is immediate by computing first order laminates. By \cite[Proposition 5.1]{DM1} these are always of the form
\begin{align*}
\lambda \mt Q \left[\begin{array}{ ccc } 1 & 0 \\ 0 & 1+\delta \end{array}\right] + (1-\lambda )\mt Q\left[\begin{array}{ ccc } 1 & 0 \\ 0 & 1-\delta \end{array}\right]
= \mt Q \left[\begin{array}{ ccc } 1 & 0 \\ 0 & 1-\delta + 2\lambda \delta \end{array}\right],
\end{align*}
with $\lambda \in (0,1),$ $\mt Q\in SO(2)$. As this yields the set on the right hand side of \eqref{eq:Klc_1} this implies that first order laminates are exactly of the same structure as the original wells. Hence, invoking Proposition 5.1 in \cite{DM1} once more, we infer that two matrices in $K^{lc,1}$ also have only one rank-one connection, and that 
$\mt Q \mt A-\mt B= \vc a \otimes \vc n$ can be satisfied by $\mt A = \left[\begin{array}{ ccc } 1 & 0 \\ 0 & 1-\delta +2 \lambda_1 \delta \end{array}\right],\mt B = \left[\begin{array}{ ccc } 1 & 0 \\ 0 & 1-\delta +2 \lambda_2 \delta \end{array}\right] \in K^{lc,1}$, $\lambda_1, \lambda_2 \in [0,1], \lambda_1 \neq \lambda_2$ and some $\mt Q \in SO(2)$, $\vc a,\vc n\in\R^2$ if and only if $\mt Q=\mt{Id}$ and $\vc a\parallel\vc n \parallel \vc e_2$. Thus further lamination does not provide a larger lamination convex hull.
\end{proof}

We use a trick due to Ball and James to characterise the quasiconvex hull of $K$:

\begin{lem}
\label{lem:Kqc}
Let $K$ be as above, then 
\begin{align*}
K^{qc} = K^{lc} = \bigcup\limits_{\lambda \in [0,1]} SO(2)\left[\begin{array}{ ccc } 1 & 0 \\ 0 & 1-\delta +2\lambda \delta \end{array}\right].
\end{align*}
\end{lem}

\begin{proof}
Let us assume without loss of generality that $\delta\geq0.$ We start by considering the function $f:\R^{2\times 2}\to \R$ given by
\begin{align*}
f(\mt F):=
\begin{cases}
\left(\det(\mt F)\right)^{-1},\qquad &\mbox{ if } \det(\mt F) >(1-\delta),\\
-\det(\mt F)(1-\delta)^{-2} + 2(1-\delta)^{-1},\qquad &\mbox{ if } \det(\mt F) \leq (1-\delta),
\end{cases}
\end{align*}
which is polyconvex, i.e., a convex function of the determinant and the minors.
As a consequence $f$ is also quasiconvex (see \cite{M1, D}), and hence
\begin{align*}
\sup_{\mt F\in K^{qc}} \det(\mt F)&\leq \sup_{\mt F\in K} \det (\mt F) \leq 1+\delta,\\
\sup_{\mt F\in K^{qc}} f(\mt F)&\leq \sup_{\mt F\in K} f(\mt F) \leq (1-\delta)^{-1},
\end{align*}
that means, $1-\delta \leq \det(\mt F) \leq 1+\delta$ for any $\mt F\in K^{qc}.$ Now, we note that the functions
\begin{align*}
\R^{2\times 2} \ni \mt F \mapsto |\mt F \vc e_1| \in \R,
 \qquad \R^{2\times 2} \ni \mt F \mapsto |f(\mt F)\Cof(\mt F) \vc e_1|
\end{align*}
are quasiconvex. 
Indeed, the norm is a convex function,  $\mt F \mapsto \mt F \vc e_1$ is affine and $\mt F \mapsto f(\mt F)\Cof(\mt F) \vc e_1$ is polyconvex. Since by the observations from above $\mt F\in K^{qc},$ $1-\delta\leq \det(\mt F) \leq 1+\delta$, we obtain that $|f(\mt F)\Cof(\mt F) \vc e_1| = |\mt F^{-T} \vc e_1|$.
As for all $\mt F\in K$ we have $|\mt F \vc e_1|=1 = |\mt F^{-T} \vc e_1|$, we also infer that
\begin{align*}
\sup\limits_{\mt F\in K^{qc}}|\mt F \vc e_1| &\leq \sup\limits_{\mt F\in K}|\mt F \vc e_1| = 1, \\ 
\sup\limits_{\mt F\in K^{qc}}|\mt F^{-T} \vc e_1| &\leq \sup\limits_{\mt F\in K}|\mt F^{-T} \vc e_1| = 1.
\end{align*}
As a consequence, for all $\mt F \in K^{qc}$ we have
\begin{align*}
|\mt F \vc e_1 - \mt F^{-T} \vc e_1 |^2 = |\mt F \vc e_1|^2 + |\mt F^{-T} \vc e_1|^2 - 2 \lbr \mt F \vc e_1, \mt F^{-T} \vc e_1 \rbr
= |\mt F \vc e_1|^2 + |\mt F^{-T} \vc e_1|^2 - 2
\leq 0,
\end{align*}
which implies that $|\mt F \vc e_1 - \mt F^{-T} \vc e_1|^2 = 0$ for all $\mt F \in K^{qc}$. Thus,
\begin{align*}
\mt F^T \mt F \vc e_1 = \vc e_1,
\end{align*}
which together with symmetry yields
$\mt F^T \mt F = \left[\begin{array}{ ccc } 1 & 0 \\ 0 & b\end{array}\right]$. Finally, $b\in [(1-\delta)^2, (1+\delta)^2]$ by the condition on the determinant of $\mt F\in K^{qc}$.
\end{proof}

\begin{rmk}
From the argument in the proof of Lemma \ref{lem:Kqc} we actually have that 
$$
K^{pc} \subset \bigcup\limits_{\lambda \in [0,1]} SO(2)\left[\begin{array}{ ccc } 1 & 0 \\ 0 & 1-\delta +2\lambda \delta \end{array}\right] = K^{lc},
$$
where $K^{pc}$ denotes the polyconvex hull of $K$ (see e.g., \cite{M1}). As, in general, $K^{lc}\subset K^{rc}\subset K^{qc}\subset K^{pc}$, where $K^{rc}$ denotes the rank-one-convex hull of $K$ (see e.g., \cite{M1}) we infer that, in the case where $K$ is given by \eqref{defK1},
$$
K^{lc}= K^{rc}= K^{qc} = K^{pc}.
$$
\end{rmk}

Finally, we prove that all gradient Young measures associated with the phase transformation \eqref{defK1} are (unique) simple laminates. In particular, no convex integration solutions exist for this phase transformation.

\begin{prop}
\label{lem:laminates}
Let $\Omega \subset \R^n$ be open, bounded, Lipschitz.
Let $\vc y^{j} \rightharpoonup \vc y$ in $W^{1,\infty}(\Omega)$ with $\vc y^{j}\in \mathcal{A}_{F}:= \{\vc y \in W^{1,\infty}(\Omega, \R^n): \vc y|_{\partial \Omega} = \mt F \vc x \}$. Let $\nu_{\vc x}$ denote the associated gradient Young measure and assume that 
\begin{align*}
\supp\, \nu_{\vc x} \in K,
\end{align*}
where $K$ is as in \eqref{defK1}. Then, there exists $\mt Q \in SO(2)$ and $\lambda\in[0,1]$ such that
\begin{align*}
\nu_{\vc x} = \lambda \delta_{\mt Q \mt F_1} + (1-\lambda)\delta_{\mt Q \mt F_2}.
\end{align*}
This representation is unique.
\end{prop}

In order to observe this, we follow the argument of Theorem 7.1. in \cite{BJ92}.

\begin{proof}
By definition we have $\mt F \in K^{qc}$ which in particular entails that $(\mt F^T \mt F)_{11}=1$. Further, by Lemma \ref{lem:lc_hull} there exists $\mt Q \in SO(2)$ such that 
\begin{align}
\label{eq:F}
\mt F = \lambda \mt Q \mt F_1 + (1-\lambda) \mt Q\mt F_2.
\end{align}
By the fact that the identity map is a Null-Lagrangian, we have $\mt F = \lbr \nu_{\vc x}, \mt A \rbr $, where $\lbr f(\mt A),  \nu_x \rbr := \int\limits_{\R^{2\times 2}} f(\mt A) d\nu_{\vc x}(\mt A) $. Denoting by $\overline{\langle\nu_{\vc x},f(\mt A) \rangle}= |\Omega|^{-1} \int_\Omega\int_{\R^{2\times 2}}f(\mt A) \,d\nu_{\vc x}(\mt A)\, d\vc x$ the average integration in $\vc x$ of $\langle\nu_{\vc x},f(\mt A) \rangle$, we obtain from Jensen's inequality and the observation that $(\mt A^T \mt A)_{11}=1$ for all $\mt A \in K$
\begin{align*}
1 = (\mt F^T \mt F)_{11} = (\overline{\lbr \nu_{\vc x}, \mt A \rbr}^T \overline{\lbr \nu_{\vc x}, \mt A \rbr})_{11}
\leq \lbr \nu_{\vc x}, (\mt A^T \mt A)_{11} \rbr = 1.
\end{align*}
Thus, 
\begin{align*}
\overline{\lbr \nu_{\vc x}, [(\mt A^T \mt A)_{11}-(\overline{\lbr \nu_{\vc x}, \mt A\rbr}^T\overline{\lbr \nu_{\vc x}, \mt A\rbr} )_{11}] \rbr } =0,
\end{align*}
which is equivalent to 
\begin{align*}
\overline{\lbr \nu_{\vc x}, [(\mt A-\mt F)^T (\mt A-\mt F)]_{11}\rbr} =\overline{\lbr \nu_{\vc x}, |(\mt A - \mt F) \vc e_1|^2\rbr} = 0 .
\end{align*}
As a consequence, 
\begin{align*}
(\mt A-\mt F) \vc e_1 = 0 \mbox{ for all }\mt A \in \supp \,\nu_{\vc x}.
\end{align*}
By \eqref{eq:F} and by $\mt F_1 \vc e_1 = \vc e_1 = \mt F_2 \vc e_1$ we obtain that $\mt A \vc e_1 = \mt Q \vc e_1$ for some $\mt Q \in SO(2)$. Therefore, $\supp\, \nu_{\vc x} \subset K$ only contains the elements $\mt Q \mt F_1$ and $\mt Q \mt F_2$. In particular,
\begin{align*}
\nu_{\vc x} = \lambda(\vc x) \mt Q \mt F_1 + (1-\lambda(\vc x)) \mt Q \mt F_2,
\end{align*}
where $\lambda: \Omega \rightarrow [0,1]$.
Using that $\nabla \vc y(\vc x) = \lambda(\vc x) \mt Q \mt F_1 + (1-\lambda(\vc x))\mt  Q \mt F_2$, we obtain that $\frac{\p}{\p x_1} \vc y = \mt F \vc e_1$, and thus that the function $ \vc z(\vc x) = \vc y(\vc x)-\mt F\vc x$ only depends on $x_2$. But since $\vc z(\vc x)=0$ on $\partial \Omega$ this implies that $\vc z(\vc x)=0$. As a consequence, $\lambda(\vc x)=\lambda \in (0,1)$ which concludes the argument.
\end{proof}

\appendix
\section{Notation}
In the convex integration algorithm described in Section \ref{sec:Algorithm}, and in the constructions of Section \ref{sec:Covering} the notation may sometimes be complex and difficult to remember throughout the paper. For this reason, in this section we summarise the main quantities which come into play.
\begin{itemize}
\item $\dist(\mathcal{A},\mathcal{B}):=\inf_{s\in\mathcal{A}}\inf_{t\in\mathcal{B}}|s-t|$ is the distance between two sets $\mathcal{A},\mathcal{B}$ (we remark that with slight abuse of notation this does \emph{not} yield a metric); 
\item $B(t,r)$ denotes the open ball centred at $t$ and of radius $r$;
\item $|\mt F|$ given $\mt F\in\R^{2\times 2}$ denotes the Frobenius norm on matrices defined as $|\mt F|^2 := \tr ( \mt F^T\mt F)$;
\item $|\vc m|$ denotes the Euclidean norm of a vector $\vc m\in\R^2$;
\item $\mt C_{\mt F} = \mt C(\mt F) = \mt F^T \mt F$;
\item $K$ is the two-well set defined in \eqref{defK}, set of our differential inclusion;
\item $K^{qc}$ is the quasiconvex hull of $K$ (see \eqref{def qc});
\item $\mathcal{G}:= \{\mt C\in\R^{2\times 2}_{Sym^+}\colon \mt C=\mt F^T\mt F,\, \mt F\in K^{qc}\}$ is the set of Cauchy-Green tensors obtained by deformation gradients in $K^{qc}$;
\item $\mathcal{V}_k$ is a sequence of sets in $K^{qc}$ that forms an in-approximation for $K$ (see Definition \ref{defi:cover});
\item $\Omega_{h,\mt R}$ is a rhombus whose axes are of length $2,2h$ rotated by a rotation $\mt R$. Thanks to Lemma \ref{lem:Conti}, at every iteration of the convex integration algorithm we can replace the deformation gradient in rotated and rescaled versions of $\Omega_{h,\mt R}$ with a deformation gradient which is closer to (but not in) $K$. $\mt R$ depends on the new deformation gradient as well as the old boundary values;
\item $\mathcal{A}_g$: Given a set $\mathcal{A}\in\R^2$, we denote by $\mathcal{A}_g$ its subset, where, at a given iteration of the convex integration algorithm, has been replaced with a deformation gradient closer (see Lemma \ref{lem:A3} and Lemma \ref{lem:A4} for the precise notion of close) to $K$; 
\item $\mathcal{A}_g^{\star}$: Given a set $\mathcal{A}\in\R^2$, $\mathcal{A}_g^{\star}$ is the subset of $\mathcal{A}_g$ where, at a given iteration of the convex integration algorithm, the deformation gradient is replaced by a close deformation gradient (see Lemma \ref{lem:A4} for details);
\item $|f|_{BV(\Omega)}$ denotes the $BV-$seminorm of a function $f\in L^1(\Omega)$ and is given by $|f|_{BV(\Omega)}:=\sup\{\int_\Omega f \phi \, d\vc x\colon \phi\in C^1_c(\Omega),\,|\phi|\leq 1\}$;
\item $\Per(\mathcal{A})$ denotes the perimeter of a measurable set $\mathcal{A}\subset\R^2$ in $\Omega$. It is defined by $\Per(\mathcal{A}):=|\chi_{\mathcal{A}}|_{BV(\Omega)}$ where $\chi_\mathcal{A}$ is the indicator function on the set $\mathcal{A}$. (see e.g., \cite{EvansGariepy}). In our case, since we are mostly considering sets $\mathcal{A}$ which are the finite union of triangles in $\R^2$, our notion will coincide with the classic notion of perimeter;
\item $\mathcal{C}$ is the set of all possible triangles (see Definition \ref{defi:cover});
\item $\mathcal C_{\mt R,h}$ is the subset of $\mathcal{C}$ which are isosceles and whose height which is also a symmetry axis coincides with the longest axis of ${\Omega_{h,\mt R}}$ (see Definition \ref{defi:cover}); 
\item $\nabla\vc u_k$ deformation gradient at the $k-$th iteration of the convex iteration algorithm;
\item $\chi_k^i$ are indicator functions of the regions where $\nabla\vc u_k$ is closer to $SO(2)\mt F_0$ and $SO(2)\mt F_0^{-1}$ respectively (see Definition \ref{def chis});
\item $\hat{\Omega}_k\subset\mathcal{C}$ is a collection of disjoint triangles covering, up to a null set, $\Omega$. We have that $\nabla\vc u_k$ is equal to a constant on every $\tilde{\Omega}\in \hat{\Omega}_k$ (see Section \ref{sec:Algorithm});
\item $\tilde{\Omega}$ is the prototypical element of $\hat{\Omega}_k\subset\mathcal{C}$;
\item $l_k\colon\hat{\Omega}_k\to\mathbb{N}$ denotes the depth of the algorithm in $\tilde{\Omega}\in \hat{\Omega}_k$. That is, the deformation gradient $\nabla\vc u_k\in \mathcal{V}_{\ell_k(\Omega)}$ in a given region $\tilde{\Omega}\in \hat{\Omega}_k$ (see Section \ref{sec:Algorithm});
\item $\mathcal{D}(\tilde\Omega)$ is the set of all descendants of $\tilde\Omega$ in the sense of Definition \ref{def descendants}.
\end{itemize}

\bibliographystyle{alpha}
\bibliography{citations1}

\end{document}